\documentclass{amsart}

\usepackage{mathtools}
\usepackage{amsfonts}
\usepackage{amssymb}
\usepackage{palatino}
\usepackage{tikz}
\usepackage{hyperref}
\usetikzlibrary{ decorations.markings}
\usepackage{float}
\DeclareMathOperator{\ad}{ad}

\DeclareFontFamily{U}{mathx}{}
\DeclareFontShape{U}{mathx}{m}{n}{<-> mathx10}{}
\DeclareSymbolFont{mathx}{U}{mathx}{m}{n}
\DeclareMathAccent{\widehat}{0}{mathx}{"70}
\DeclareMathAccent{\widecheck}{0}{mathx}{"71}

\theoremstyle{plain}

\newtheorem{theorem}{Theorem}[section]

\newtheorem{lemma}[theorem]{Lemma}

\newtheorem{problem}[theorem]{Problem}
\theoremstyle{remark}
\newtheorem{remark}[theorem]{Remark}
\makeatletter

\newcommand{\Rmnum}[1]{\expandafter\@slowromancap\romannumeral #1@}
\makeatother

\numberwithin{equation}{section}

\begin{document}
	\title{Higher order asymptotics for the nonlinear Schr\"odinger equation}
	\author{Jiaqi Liu}
	\author{Changhua Yang}
	\address[Liu]{School of mathematics, University of Chinese Academy of Sciences. No.19 Yuquan Road, Beijing China }
\email{jqliu@ucas.ac.cn}
\address[Yang]{School of mathematics, University of Chinese Academy of Sciences. No.19 Yuquan Road, Beijing China}
\email{yangchanghua22@mails.ucas.ac.cn}
	\begin{abstract}
	    In this paper we compute the higher order long time asymptotics of the defocussing nonlinear Schr\"odinger equation using the $\overline{\partial}$-nonlinear steepest descent method. We assume initial condition in weighted Sobolev space with finite order of regularity and decay.
	\end{abstract}
\maketitle

\newcommand{\dint}{\displaystyle{\int}}
\newcommand{\diff}{\partial}
\newcommand{\eps}{\varepsilon}

\newcommand{\norm}[2]{\left\Vert #1\right\Vert_{#2}}
\newcommand{\bigO}[2][ ]{\mathcal{O}_{#1} \left( #2 \right)}

\newcommand{\C}{\mathbb{C}}

\newcommand{\ba}{\breve{a}}
\newcommand{\bb}{\breve{b}}
\newcommand{\bbC}{\mathbb{C}}
\newcommand{\bbR}{\mathbb{R}}
\newcommand{\bfL}{\mathbf{L}}
\newcommand{\bfP}{\mathbf{P}}
\newcommand{\calT}{\mathcal{T}}
\newcommand{\wcalT}{\widetilde{\mathcal{T}}}
\newcommand{\calR}{\mathcal{R}}
\newcommand{\wcalR}{\widetilde{\mathcal{R}}}
\newcommand{\tk}{\tilde{k}}

\newcommand{\zbar}{\overline{z}}
\newcommand{\diagmat}[2]
{
\left(
	\begin{array}{cc}
		{#1}	&	0	\\
		0		&	{#2}
		\end{array}
\right)
}
\newcommand{\offdiagmat}[2]
{
\left(
	\begin{array}{cc}
		0			&		{#1} 	\\
		{#2}		&		0
		\end{array}
\right)
}

\newcommand{\SixMatrix}[6]
{
\begin{figure}
\centering
\caption{#1}
\vskip 15pt
\begin{tikzpicture}
[scale=0.7]
%%%%%%%%%%%%%%%%%%%%%%%%
%
%		Contours and shading
%
%%%%%%%%%%%%%%%%%%%%%%%%
\draw[thick]	 (-4,0) -- (4,0);
\draw[thick] 	(-4,4) -- (4,-4);
\draw[thick] 	(-4,-4) -- (4,4);
%\path[fill=gray,opacity=0.25] (0,0) -- (-4,4) -- (4,4) -- (0,0);
%\path[fill=gray,opacity=0.25] (0,0) -- (-4,-4) -- (4,-4) -- (0,0);
%%%%%%%%%%%%%%%%%%%%%%%%
%
%		Origin
%
%%%%%%%%%%%%%%%%%%%%%%%%
\draw	[fill]		(0,0)						circle[radius=0.075];
\node[below] at (0,-0.1) 				{$z_0$};
%%%%%%%%%%%%%%%%%%%%%%%%
%
%		 Sector labels
%
%%%%%%%%%%%%%%%%%%%%%%%%
\node[above] at (3.5,2.5)				{$\Omega_1$};
\node[below]  at (3.5,-2.5)			{$\Omega_6$};
\node[above] at (0,3.25)				{$\Omega_2$};
\node[below] at (0,-3.25)				{$\Omega_5$};
\node[above] at (-3.5,2.5)			{$\Omega_3$};
\node[below] at (-3.5,-2.5)			{$\Omega_4$};
%%%%%%%%%%%%%%%%%%%%%%%%
%
%		 Identity matrices
%
%%%%%%%%%%%%%%%%%%%%%%%%
\node[above] at (0,1.25)				{$\twomat{1}{0}{0}{1}$};
\node[below] at (0,-1.25)				{$\twomat{1}{0}{0}{1}$};
%%%%%%%%%%%%%%%%%%%%%%%%
%
%		Remaining Matrices
%
%%%%%%%%%%%%%%%%%%%%%%%%
\node[right] at (1.20,0.70)			{$#3$};
\node[left]   at (-1.20,0.70)			{$#4$};
\node[left]   at (-1.20,-0.70)			{$#5$};
\node[right] at (1.20,-0.70)			{$#6$};
\end{tikzpicture}
\label{#2}
\end{figure}
}
\newcommand{\twomat}[4]
{
\left(
	\begin{array}{cc}
		{#1}	&	{#2}	\\
		{#3}	&	{#4}
		\end{array}
\right)
}

\section{Introduction}
In this paper we calculate the higher order 
long-time asymptotics of  solutions  to  the defocussing  nonlinear Schr\"odinger equation (NLS):
\begin{equation}
\label{NLS}
iq_t + q_{xx} - 2|q|^2q=0, \qquad (x, t)\in (\bbR, \bbR^+).
\end{equation}
As a canonical example of  dispersive PDEs, the long time asymptotics of \eqref{NLS} has been extensively studied and the optimal result of the first order expansion is given in \cite{DMM}:
\begin{theorem}
    For $q_0=q(x,0)\in H^{1,1}$, if  $q(x,t)$ solves Equation \eqref{NLS}, then $q(x,t)$ admits the following asymptotic expansion:
  \begin{equation}
  \label{q:dmm}
      q(x,t)=q_{as}(x,t)+\mathcal{O}(t^{-3/4})
  \end{equation}
  where
  \begin{align}
      &q_{as}(x,t)=\mathrm{e}^{-\mathrm{i} \omega\left(z_0\right)} \mathrm{e}^{-2 \mathrm{i} t\theta\left(z_0 ; z_0\right)} c\left(z_0\right)^{-2}\left(2 t^{1 / 2}\right)^{-2 \mathrm{i} \nu\left(z_0\right)}\left[\frac{1}{2} t^{-1 / 2} \beta\left(\left|r\left(z_0\right)\right|\right)\right],\\
      \label{nuz0}
      &\nu\left(z_0\right) =-\frac{1}{2 \pi} \ln \left(1-\left|r\left(z_0\right)\right|^2\right),\\
      &\nonumber\left|\beta\right|^2=2 \nu\left(z_0\right),\\
      &\arg (\beta)=\frac{\pi}{4}+\frac{1}{2 \pi} \ln (2) \ln \left(1-\left|r\left(z_0\right)\right|^2\right)-\arg \left(\Gamma\left(\frac{i}{2 \pi} \ln \left(1-\left|r\left(z_0\right)\right|^2\right)\right)\right),\\
\label{omegaz0}&\omega\left(z_0\right)=\arg\left(r(z_0)\right),\\ 
\label{thetaz0} &\theta\left(z_0 ; z_0\right)=-2 z_0^2.
      \end{align}
      Here $\Gamma $ is the gamma-function and $r$ is the reflection coefficient given by \eqref{reflection} and $z_0$ is the stationary point $x/4t$ and $c(z_0)$ is given by \eqref{cz0}.
\end{theorem}
In \cite{SA}, large $t$ expansions of the solution to Equation \eqref{NLS} up to infinite orders is first studied. In \cite{DZ94}, using the nonlinear steepest descent method developed in \cite{DZ93}, the authors calculated the full asymptotic expansion of $q(x,t)$ in $t$ assuming Schwarz initial condition. See \cite[Theorem 1.10 ]{DZ94} for the explicit form of the following formula:
\begin{equation}
    \label{asy:full}
    q(x, t)=e^{i x^2 / 4 t-i \nu \ln t}\left(\sum_{p=1}^N \dfrac{\sum_{q=0}^{p-1} u_{p q}\left(z_0\right)(\log t)^q}{t^{p / 2}}+\mathcal{O}\left(\frac{(\ln t)^N}{t^{(N+1) / 2}}\right)\right)
\end{equation}
where $u_{p,q}=0$ for $p$ even.

The goal of this paper is to extend the $\overline{\partial}$-nonlinear steepest descent method developed in \cite{DM} and \cite{DMM} to study higher order asymptotics of \eqref{NLS} whose initial condition only assume finite order of smoothness and decay. The main result of our paper is the following theorem:
\begin{theorem}
\label{thm:main}
    For $q_0=q(x,0)\in H^{2,4}$, if  $q(x,t)$ solves Equation \eqref{NLS}, then $q(x,t)$ admits the following asymptotic expansion:
  \begin{equation}
  \label{formu:main}
  q(x,t)=q_{as}(x,t)+o\left(\frac{\ln t}{t}\right)
\end{equation}
where $ q_{as}(x,t)$ is given by \eqref{q:dmm}. And for the explicit derivation of the expression we refer the reader to Section \ref{sub:higher}.
\end{theorem}
\begin{remark}
    In this paper we are not pursuing for the lowest regularity assumptions on the initial condition. Instead we want to extend the $\overline{\partial}$-nonlinear steepest descent method first introduced by \cite{DM} to the study of higher order expansions. We would like to point out that in \cite{DIZ}, assuming the reflection coefficient $r$ has classical derivatives up to the fifth order, an error term of $O(\ln t/t)$ is obtained. In this paper, we show that given reflection coefficient $r\in H^{4,2}$, one can obtain a correction term of order $o(\ln t/t)$. 
\end{remark}
\begin{remark}
 Notice that our asymptotic formula matches up with \eqref{asy:full} when $N=1,2$. It is possible to carry on with the calculation in this paper to fully recover all the terms of \eqref{asy:full}. We also believe that this approach can be used  to study higher order asymptotics of other completely integrable PDEs with or without the presence of solitons.
\end{remark}
\subsection{The direct and inverse scattering transform}
To describe our approach, recall that by complete integrability,
 \eqref{NLS} is the compatibility condition for the following commutator:
\begin{align}
\label{eq: compat}
    \left[\partial_x-U, \partial_t-W\right]=0,
\end{align}
where
\begin{align*}
U & =i z \sigma+\left(\begin{array}{cc}
0 & q \\
\bar{q} & 0
\end{array}\right), \\
W & =-i z^2 \sigma-z\left(\begin{array}{cc}
0 & q \\
\bar{q} & 0
\end{array}\right)+\left(\begin{array}{cc}
-i|q|^2 & i \partial_x q \\
-i \partial_x \bar{q} & i|q|^2
\end{array}\right).
\end{align*}
To describe our approach, we recall that \eqref{NLS} generates an iso-spectral flow for the problem
\begin{equation}
\label{L}
\frac{d}{dx} \Psi = -iz \sigma_3 \Psi + U(x) \Psi
\end{equation}
where
$$ \sigma_3 = \diagmat{1}{-1}, \,\,\, U(x) = \offdiagmat{q(x)}{\overline{q(x)}}.$$
This is a standard AKNS system. If $u \in L^1(\bbR) $, Equation \eqref{L} admits bounded 
solutions for $z \in \mathbb{R}$. There exist unique solutions $\Psi^\pm$ of \eqref{L} obeying the the following space asymptotic conditions
$$\lim_{x \to \pm \infty} \Psi^\pm(x,z) e^{-ix z \sigma_3} = \diagmat{1}{1},$$
and there is a matrix $T(z)$, the transition matrix, with 
 $\Psi^+(x,z)=\Psi^-(x,z) T(z)$.
The matrix $T(z)$ takes the form
\begin{equation} \label{matrixT}
 T(z) = \twomat{a(z)}{\bb(z)}{b(z)}{\ba(z)} 
 \end{equation}
and  the determinant relation gives
$$ a(z)\ba(z) - b(z)\bb(z) = 1 $$
Combining this with the symmetry relations 
\begin{align} \label{symmetry}
\ba(z)=\overline{a( \zbar )}, \quad \bb(z) = \overline{ b(\zbar)}. 
\end{align}
we arrive at
$$|a(z)|^2-|b(z)|^2=1$$
and conclude that $a(z)$ is zero-free. By the standard inverse scattering theory, we formulate the reflection coefficient:
\begin{equation}
\label{reflection}
r(z)=\bb(z)/a(z), \quad z\in\bbR
\end{equation}
In \cite{zhou}, it is shown that for $k, j$ integers with $k\geq 0$, $j\geq 1$, the direct scattering map $\mathcal{R}$ maps $H^{k,j}(\bbR)$ onto $H^{j,k}_1=H^{j,k}(\bbR)\cap \lbrace r: \norm{r}{L^\infty} <1\rbrace$ where $H^{j,k}$ norm is given by:
\begin{equation}
\label{sp: weighted}
    \norm{q}{H^{i,j}(\bbR)}
= \left( \norm{(1+|x|^j)q}{2}^2 + \norm{q^{(i)}}{2}^2 \right)^{1/2}. 
\end{equation}
and the map $\mathcal{R}: u_0 \mapsto r$ is Lipschitz continuous. Since we are dealing with the defocussing NLS,  only the reflection coefficient $r$ is needed for the reconstruction of the solution. The long-time behavior of the solution to the NLS equation is obtained through a sequence of transformations of the following RHP:
\begin{problem}
\label{prob:DNLS.RH0}
Given $r(z) \in H^{1,1}(\bbR)$, 
find a $2\times 2$ matrix-valued function $m(z;x,t)$ on $\bbC \setminus \bbR$ with the following properties:
\begin{enumerate}
\item		$M(z;x,t) \to I$ as $|z| \to \infty$,
			\medskip
			
\item		$M(z;,x,t)$ is analytic for $z \in \mathbb{C} \setminus \bbR$ with continuous boundary values
			$$M_\pm(z;x,t) = \lim_{\varepsilon \to 0} M(z\pm i\varepsilon;x,t),$$
			\medskip
			
\item		
The jump relation $M_+(z;x,t) = M_	-(z;x,t) e^{-i\theta \ad \sigma_3} v(z)$ holds, where
			\begin{equation}
			\label{mkdv.V}
			e^{-i\theta \ad \sigma_3} v(z)	=		\twomat{1-|r(z)|^2}
											{-\overline{r(z)}e^{-2it\theta}}{r(z)e^{2it\theta}}{1}
			\end{equation}
and the real phase function $\theta$ is given by
\begin{equation}
\label{mkdv.phase}
t\theta(z;x,t) = 4tz^2+xz
\end{equation}
with stationary points
\begin{equation}
    \label{stationary pt}
    z_0={ -\dfrac{x}{4t}} .
    \end{equation}
    \end{enumerate}
\end{problem}
From the solution of Problem \ref{prob:DNLS.RH0}, we recover
\begin{align}
\label{nls.q}
q(x,t) &= \lim_{z \to \infty} 2 i z M_{12}(x,t,z)
\end{align}
where the limit is taken in $\bbC\setminus \bbR$ along any direction not tangent to $\bbR$.

\section{The Higher Order Aysmptotics}
In this section we derive the higher order asymptotic formula for \eqref{NLS}. We are not going to repeat the standard steps of $\overline{\partial}$-nonlinear steepest descent method here. Instead we start with the following scalar Riemann-Hilbert problem and derive some estimates that will play a key role in the calculation of higher order expansion.
\subsection{Conjugation}
We introduce a new matrix-valued function
\begin{equation}
\label{m1}
M^{(1)}(z;x,t) = M(z;x,t) \delta(z)^{-\sigma_3} 
\end{equation}
where $\delta(z)$  solves 
the scalar RHP 
Problem \ref{prob:RH.delta} below:
\begin{problem}
\label{prob:RH.delta}
Given $z_0 \in \bbR$ and $r \in H^{1}(\bbR)$, find a scalar function 
$\delta(z) = \delta(z; z_0)$, analytic for
$z \in \bbC \setminus (-\infty, z_0]$ with the following properties:
\begin{enumerate}
\item		$\delta(z) \to 1$ as $z \to \infty$,
\item		$\delta(z)$ has continuous boundary values $\delta_\pm(z) =\lim_{\eps \to 0} \delta(z \pm i\eps)$ for $z \in (-\infty, z_0]$,
\item		$\delta_\pm$ obey the jump relation
			$$ \delta_+(z) = \begin{cases}
											\delta_-(z)  \left(1 - \left| r(z) \right|^2 \right),	&	 z\in (-\infty, z_0]\\
											\delta_-(z), &	z \in \bbR\setminus (-\infty, z_0].
										\end{cases}
			$$
\end{enumerate}
\end{problem}
The following lemma is standard:
\begin{lemma}
\label{lemma:delta}
Suppose $r \in H^{1}(\bbR)$. Then
\begin{itemize}
\item[(i)]		Problem \ref{prob:RH.delta} has the unique solution
\begin{equation}
\label{RH.delta.sol}
\delta(z) =  e^{\chi(z)}  
\end{equation}
where 
\begin{equation}
\label{chi}
\chi(z)=\dfrac{1}{2\pi i}\int_{-\infty}^{z_0}\ln\left( \frac{1-|r(s)|^2} {s-z}\right)ds.
\end{equation}
Here we choose the branch of the logarithm with $-\pi  < \arg(z) < \pi$. 
\bigskip
\item[(ii)]
\begin{equation*}
\delta(z) =(\overline{\delta(\zbar)})^{-1}=\overline{\delta(-\zbar)}
\end{equation*}
\bigskip
\item[(iii)]
For $z\in\bbR$, $|\delta_\pm(z)|<\infty$; for $z\in \bbC\setminus\bbR$, $|\delta^{\pm 1}(z)|<\infty$
\item[(iv)]Along any ray of the form $ z_0+ e^{i\phi}\bbR^+$ with $0<\phi<\pi$ or $\pi < \phi < 2\pi$, 
$$ \left| \delta(z) - \left( {z-z_0} \right)^{iF(z_0)/2\pi} e^{\chi( z_0)}  \right| \leq C_r |z - z_0|^{1/2}.$$
 Here $$F(z_0):=\ln (1-|r(z_0)|^2)=-2\pi\nu(z_0),$$				
and the implied constant depends on $r$ through its $H^{1}(\bbR)$-norm  
				and is independent of $\pm z_0\in \bbR$.

\end{itemize}
\end{lemma}
Set 
\begin{equation}
\label{cz0}
    c(z_0): = \exp{\left( \frac{1}{2\pi i }\int_{-\infty}^{z_0}\ln(z_0-s)dF(s)\right)},
\end{equation}
we proceed to define 
\begin{equation}
\label{exp:f}
f(z;z_0) = c(z_0)\delta(z;z_0)(z-z_0)^{iF(z_0)/2\pi}.
\end{equation}
\begin{lemma}
    Suppose that  $r(z) \in H^{4, 2}(\mathbb{R})$ and there exists $0<\rho <1$ such that $|r(z)|\leq \rho$ holds for all $z\in \mathbb{R}$. Then
    \begin{itemize}
        \item[(1)] The functions $f(z,z_0)^{\pm 1}$ are uniformly bounded in the following sense
        \begin{equation}
        \label{boundf}
            \sup\limits_{\substack{z_0\in \mathbb{R}\\ \arg(z-z_0)\in (-\pi,\pi)}} \leq \frac{1}{1-\rho^{2}}.   
        \end{equation}
        \item[(2)]The functions $f(z,z_0)^{\pm 1}$ are well-defined and analytic in z for $\arg(z-z_0)\in (-\pi,\pi)$.
        \item[(3)] The functions $f^{\pm 2}(z,z_0)$ are \textit{H{\"o}lder} continuous with exponent 1/2 and there is a constant $\mu = \mu(\nu(z_0)) >0$ such that
            \begin{equation}
                |f^{\pm 2}(z,z_0)-1|\leq \mu |z-z_0|^{\frac{1}{2}}
            \end{equation}
            holds whenever $\arg(z-z_0)\in (-\pi,\pi)$.
        \item[(4)]  $f^{\pm 2}(z,z_0)$ has the following expansion:
           \begin{equation}
           \label{f}
        f^{\pm 2}(z;z_0)=\exp\left(\pm\frac{F^{'}(z_0)}{\pi i}(z-z_0)\left[\ln(z-z_0)-1\right]+o\left((z-z_0)\ln(z-z_0)\right)\right).\\
    \end{equation}
    In particular,for $r(z) \in H^{4,2}(\mathbb{R})$, $u=z_0+\rho \cos{\phi}$ and $v=\rho \sin{\phi}$, $f^{\pm2}(\frac{\rho e^{i\phi}}{t^{\frac{1}{2}}};z_0)$ has the following expansion for large $t$:
    \begin{align}
    \label{functionf}
        f^{\pm 2}\left(\frac{\rho e^{i\phi}}{t^{{1}/{2}}};z_0\right)&=1\pm\frac{F'(z_0)}{\pi i}\left(\frac{\rho e^{i\phi}}{t^{\frac{1}{2}}}\right)\left[-\frac{1}{2}\ln t+\ln\rho e^{i\phi}-1\right]+{o}\left(\frac{\ln t}{t^{\frac{1}{2}}}\right)\\
        \nonumber
        &=1\pm \frac{F'(z_0)}{\pi i}\left(\frac{\rho e^{i\phi}}{t^{\frac{1}{2}}}\right)\left[-\frac{1}{2}\ln t\right]+{o}\left(\frac{\ln t}{t^{\frac{1}{2}}}\right)
    \end{align}
 \end{itemize}
\begin{proof}
The proof of (1)-(3) are standard and can be found in \cite{DMM}. We just prove (4). Since $f(z;z_0) = c(z_0)\delta(z;z_0)(z-z_0)^{iF(z_0)/2\pi}$, we have
\begin{align}
    c(z_0)\delta(z;z_0) &= c(z_0)\exp{\left(\int_{-\infty}^{z_0}\frac{1}{2\pi i }\frac{F(s)}{s-z}ds\right)} \\
    \nonumber
    &= c(z_0)\exp{\left(\frac{1}{2\pi i} \int_{-\infty }^{z_0} F(s)d\ln(z-s)\right)}\\\nonumber
    & = c(z_0)\exp{\left(\frac{1}{2\pi i}\left [ F(s)\ln(z-s) \right ] |_{-\infty}^{z_0}+\frac{1}{2\pi i}\int_{-\infty }^{z_0}F^{'}(s)d\left [ (\ln(z-s)-1)(z-s) \right ]\right)} \\
    \nonumber
    &=c(z_0)(z-z_0)^{-iF(z_0)/2\pi}\exp\left(\frac{1}{2\pi i}\int_{-\infty }^{z_0}F^{'}(s)d\left [ (\ln(z-s)-1)(z-s) \right ]\right).
\end{align}
Then
$$f(z;z_0) = c(z_0)\exp{\left(\frac{1}{2\pi i}\int_{-\infty }^{z_0}F^{'}(s)d\left [ (\ln(z-s)-1)(z-s) \right ]\right)}.$$ 
Note that 
\begin{align}
c(z_0) &:= \exp{\left( \frac{1}{2\pi i }\int_{-\infty}^{z_0}\ln(z_0-s)dF(s
)\right)}\\
&=\exp{\left( \frac{1}{2\pi i }\int_{-\infty}^{z_0}\ln(z_0-s)F'(s)ds\right)}\\ \nonumber
& = \exp{\left(- \frac{1}{2\pi i }\int_{-\infty}^{z_0}F'(s)d\left [ (\ln(z_0-s)-1)(z_0-s) \right ]
)\right)}\\
\nonumber
&=\exp{\left(\frac{1}{2\pi i }\int_{-\infty}^{z_0}F''(s)\ln(z_0-s)-1)(z_0-s)ds
)\right)}.
\end{align}
%This will eliminate the $-\infty$ term when integral by part.
Then
\begin{align}
    f(z;z_0) = c(z_0)\exp{\left( (z-z_0)\left [ \frac{1}{2\pi i}F^{'}(z_0)\ln(z-z_0)-\frac{1}{2\pi i}F^{'}(z_0) \right ] -\frac{1}{2\pi i}\int_{-\infty }^{z_0}F^{''}(s)\left [ (\ln(z-s)-1)(z-s) \right ]ds\right)}. \nonumber
\end{align}
%where $\widehat{c}(z_0) = \exp{\left(\frac{1}{2\pi i }\int_{-\infty}^{z_0}\left [ (\ln(z_0-s)-1)(z_0-s)F''(s)ds \right ]
%)\right)}$\\
As for the $o((z-z_0)\ln(z-z_0) )$ term, according to  L'Hospital's rule,we have 
\begin{align}
    &\lim_{z\to z_0}\frac{\int_{-\infty}^{z_0}F''(s) \left(\ln(z-s)-1\right)(z-s)-F''(s) \left(\ln(z_0-s)-1\right)(z_0-s)ds}{(z-z_0)\ln{(z-z_0)}}\\
    \nonumber
    &\lesssim \lim_{z\to z_0}\frac{|\int_{-\infty}^{z_0}F''(s) \ln(z-s)ds|}{|\ln(z-z_0)+1|} = 0.
\end{align}
We mention that it is easy to check that given $r(z)\in H^{4,2}$ which is a consequence of the Sobolev mapping property established in \cite{zhou}, the integral
$$ \int_{-\infty}^{z_0}F''(s) \ln(z-s)ds$$
is finite as $z\to z_0$.
Replace $\rho$ with ${\rho}/{t^{{1}/{2}}}$ we arrive at \eqref{functionf}.
\end{proof}
\end{lemma}
Recall in \cite{DMM}, the $\overline{\partial}$- nonlinear steepest descent method gives the following expression for the solution to Equation \eqref{NLS}:
\begin{align}
    q(x,t)&=\mathrm{e}^{-\mathrm{i} \omega\left(z_0\right)} \mathrm{e}^{-2 \mathrm{i} t\theta\left(z_0 ; z_0\right)} c\left(z_0\right)^{-2}\left(2 t^{1 / 2}\right)^{-2 \mathrm{i} v\left(z_0\right)}\left[2 \mathrm{i} E_{1,12}(x, t)+\frac{1}{2} t^{-1 / 2} \beta\left(\left|r\left(z_0\right)\right|\right)\right].
\end{align}

In order to derive the higher order asymptotics, by \cite[(132)]{DMM} we deal with the large $t$ expansion of the following term:
\begin{align}
\label{q1}
q^{(1)}(x,t)&= \frac{2i}{\pi}e^{-i\omega(z_0)}e^{-it\theta(z_0, z_0)}c(z_0)^{-2}(2t^{\frac{1}{2}})^{-2i\nu(z_0)}\iint_{\mathbb{R}^2}W_{12}(u,v;x,t)d\mathcal{A}(u,v)
\end{align}
where $W_{12}$ is the $(1,2)$ entry of the matrix-valued function given by \cite[(82)]{DMM}. We now explicitly compute the double integral in \eqref{q1}. We will follow the notations in \cite{DMM}.

{
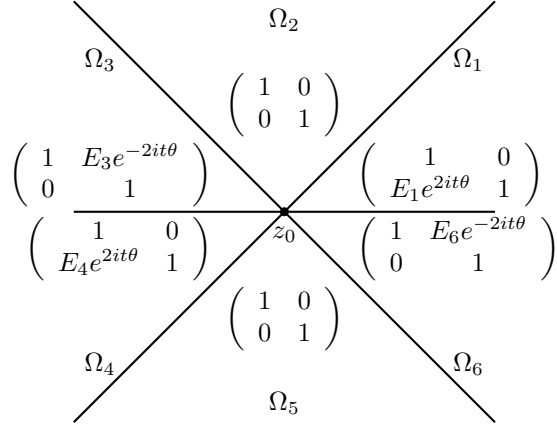
\SixMatrix{The Matrix  ${E}$ near $z_0$}{fig R-2+}
	{\twomat{1}{0}{E_1 e^{2it\theta}}{1}}
	{\twomat{1}{E_3 e^{-2it\theta}}{0}{1}}
	{\twomat{1}{0}{E_4 e^{2it\theta}}{1}}
	{\twomat{1}{E_6 e^{-2it\theta}}{0}{1}}
}
Through standard contour deformation of the nonlinear steepest descent, the complex plane is divided into the following six sectors:
\begin{equation}
\begin{aligned}
    &\Omega_{1}:0<\arg(z-z_0)<\frac{\pi}{4}; \quad \Omega_{2}:\frac{\pi}{4}<\arg(z-z_0)<\frac{3\pi}{4};\quad \Omega_{3}:\frac{3\pi}{4}<\arg(z-z_0)<\frac{5\pi}{4}\\ &\Omega_{4}:-\pi<\arg(z-z_0)<-\frac{3\pi}{4}; \quad \Omega_{5}:-\frac{3\pi}{4}<\arg(z-z_0)<\frac{\pi}{4}; \quad \Omega_{6}:-\frac{\pi}{4}<\arg(z-z_0)<0\\  
\end{aligned}
\end{equation}
\begin{align}
E_1(u,v)&=\cos(2\arg(u+iv-z_0))f^{-2}(u+iv;z_0)r(u)e^{-i\omega(z_0)}\nonumber\\
&\quad +(1-\cos(2\arg(u+iv-z_0)))|r(z_0)|, \nonumber\\
E_3(u,v)&=-\cos(2\arg(u+iv-z_0))f^{2}(u+iv;z_0)\frac{\overline{r(u)}e^{i\omega(z_0)}}{1-|r(u)|^{2}}\nonumber\\
&\quad -(1-\cos(2\arg(u+iv-z_0)))\frac{{r(z_0)}}{1-|r(z_0)|^{2}},\nonumber\\
E_4(u,v)&=\cos(2\arg(u+iv-z_0))f^{-2}(u+iv;z_0)\frac{\overline{r(u)}e^{-i\omega(z_0)}}{1-|r(u)|^{2}}\nonumber\\
&\quad +(1-\cos(2\arg(u+iv-z_0)))\frac{{r(z_0)}}{1-|r(z_0)|^{2}},\nonumber\\
E_6(u,v)& = -\cos(2\arg(u+iv-z_0))f^{2}(u+iv;z_0)\overline{r(u)}e^{i\omega(z_0)}\nonumber\\
&\quad -(1-\cos(2\arg(u+iv-z_0)))|r(z_0)|.\nonumber
\end{align}
We rewrite $u,v$ in polar coordinates: $u=z_0+\rho \cos{\phi}$ and $v=\rho sin{\phi}$. Recall that the $\overline{\partial}$ operator  takes the form
\begin{align}
\overline{\partial} :=\frac{1}{2}(\frac{\partial}{\partial u}+i\frac{\partial}{\partial v}) = \frac{e^{i\phi}}{2}(\frac{\partial}{\partial \rho}+\frac{i}{\rho}\frac{\partial}{\partial \phi}).
\end{align}
Setting $\phi=\arg(u+iv-z_0)$, we have
\begin{align}
    \overline{\partial}\cos{\left(2\arg(u+iv-z_0)\right)}=\frac{ie^{i\phi}}{2\rho}\frac{d}{d\phi}\cos{2\phi}=-\frac{ie^{i\phi}}{\rho}\sin{2\phi}
\end{align}
Since 
\begin{align}
    &\overline{\partial}f^{\pm2}(u+iv;z_0)=0,\\
    &\overline{\partial} r(z_0+u)=\frac{1}{2}r'(z_0+u),
\end{align}
 In different sectors of the complex plane, direct computation gives
\begin{subequations}
    \begin{align}
   \overline{\partial}E_1(u,v) &=\frac{ie^{i\phi}}{\rho}\sin{2\phi}|r(z_0)|-\frac{ie^{i\phi}}{\rho}\sin(2\phi)f^{-2}(u+iv;z_0)r(u)e^{-i\omega(z_0)}\\
   \nonumber
   &\quad +\frac{1}{2}\cos{2\phi}f^{-2}(u+iv;z_0)r'(u)e^{-i\omega(z_0)}, \quad z\in \Omega_1
  ,\\
   \overline{\partial}E_3(u,v)&=-\frac{ie^{i\phi}}{\rho}\sin{2\phi}\frac{|r(z_0)|}{1-|r(z_0)|^{2}}+\frac{ie^{i\phi}}{\rho}\sin{2\phi}f^{2}(u+iv;z_0)\frac{\overline{r(u)}e^{i\omega(z_0)}}{1-|r(u)|^{2}}\\
   \nonumber
   &\quad -\frac{1}{2}\cos{2\phi}f^{2}(u+iv;z_0)e^{i\omega(z_0)}\frac{\overline{r'(u)}+(\overline{r^{2}(u)})r'(u)}{(1-|r(u)|^2)^2}, \quad z\in \Omega_3\\
   \overline{\partial}E_4(u,v)&=\frac{ie^{i\phi}}{\rho}\sin{2\phi}\frac{|r(z_0)|}{1-|r(z_0)|^{2}}-\frac{ie^{i\phi}}{\rho}\sin{2\phi}f^{-2}(u+iv;z_0)\frac{{r(u)}e^{-i\omega(z_0)}}{1-|r(u)|^{2}}\\
   \nonumber
   &\quad +\frac{1}{2}\cos{2\phi}f^{-2}(u+iv;z_0)e^{-i\omega(z_0)}\frac{{r'(u)}+r^{2}(u)\overline{r'(u)}}{(1-|r(u)|^2)^2}, \quad z\in \Omega_4\\
   \overline{\partial}E_6(u,v)& =-\frac{ie^{i\phi}}{\rho}\sin{2\phi}|r(z_0)|+\frac{ie^{i\phi}}{\rho}\sin{2\phi}f^{2}(u+iv;z_0)\overline{{r(u)}}e^{i\omega(z_0)}\\
   \nonumber
   &\quad -\frac{1}{2}\cos{2\phi}f^{2}(u+iv;z_0)\overline{r'(u)}e^{i\omega(z_0)}, \quad z\in \Omega_6.
    \end{align}
\end{subequations}
We now analyze the higher order $t$ asymptotics of \eqref{q1} in each region.
\subsection{Large $t$ expansion in $\Omega_1$}
In $\Omega_1 (0<\arg(z-z_0)<\frac{\pi}{4})$, we have that
\begin{align}
W_{12}=-c_{1}(z_0)e^{-2it(u+iv-z_0)^{2}}U^2(a;y)(-\overline{\partial}E_1(u,v)e^{2it[\theta(u+iv;z_0)-\theta(z_0;z_0)]}).
\end{align}
where 
\begin{align}
   & A_{2}^{(0)} = \frac{1}{\sqrt{2}}\left(1-|r(z_0)|^{2}\right)^{-1/8}e^{-3\pi i/4}\exp\left(-i\frac{1}{4\pi}\ln2F(z_0)\right)\\
    &c_{1}(z_0)=\beta^{2}(A_{2}^{(0)})^{2}.
\end{align}
The following lemma will be needed. For details see \cite{nist}.
\begin{lemma}
    $U(a,y)$ is bounded and it has the following asymptotic expansion
\begin{align}
    U(a,y)&=e^{-\frac{1}{4}y^{2}}(U(a,0)u_1(a,y)+U'(a,0)u_2(a,y))\\
    \nonumber
   & := e^{-\frac{1}{4}y^{2}}A(a,y).
\end{align}
where
\begin{align}
\nonumber
    &u_1(a,y)=(1+(a+\frac{1}{2})\frac{y^2}{2!}+(a+\frac{1}{2})(a+\frac{5}{2})\frac{y^4}{4!}+\cdots),\\
\nonumber     &u_2(a,y)=(y+(a+\frac{3}{2})\frac{y^3}{3!}+(a+\frac{3}{2})(a+\frac{7}{2})\frac{y^5}{5!}+\cdots),\\
     &U(a,0) = \frac{\sqrt{\pi}}{2^{\frac{a}{2}+\frac{1}{4}}\Gamma\left(\frac{3}{4}+\frac{1}{2}a\right)},\nonumber\\
     &U'(a,0)=-\frac{\sqrt{\pi}}{2^{\frac{a}{2}-\frac{1}{4}}\Gamma\left(\frac{1}{4}+\frac{1}{2}a\right)}.\nonumber
\end{align}
where
\begin{align*}
  &a=\frac{1}{2}(1+i|\beta|^2),\quad y=2\sqrt{2}e^{-\frac{\pi i}{4}}t^{\frac{1}{2}}(u+iv-z_0)\\
&\frac{dA(a,y)}{dy} = -(a+\frac{1}{2})A(a+1,y) 
\end{align*}
\end{lemma}

Direct computation gives
\begin{align}
\label{Omega1}
  \iint_{\Omega_{1}}W_{12}(u,v;x,t)d\mathcal{A}(u,v)  &=-c_{1}(z_0)\int_{0}^{+\infty}d\rho \int_{0}^{\frac{\pi}{4}}U^2(a;y)\rho (-\overline{\partial}E_1)d\phi \\
    \nonumber
   & =c_{1}(z_0)\int_{0}^{+\infty}U^2(a;y)d\rho\int_{0}^{\frac{\pi}{4}}ie^{i\phi}\sin{2\phi} |r(z_0)| d\phi \\
   \nonumber
   &\quad -c_{1}(z_0)\int_{0}^{+\infty}U^2(a;y)d\rho \int_{0}^{\frac{\pi}{4}} ie^{i\phi}\sin{2\phi} f^{-2}(\rho e^{i\phi};z_0)r(z_0+\rho \cos\phi)e^{-i\omega(z_0)}d\phi\\ 
   \nonumber
	&\quad +c_{1}(z_0)\int_{0}^{+\infty}U^2(a;y)d\rho \int_{0}^{\frac{\pi}{4}}\frac{\rho}{2}\cos{2\phi} f^{-2}(\rho e^{i\phi};z_0)r'(z_0+\rho \cos\phi)e^{-i\omega(z_0)}d\phi\\
    \nonumber
      &:=-(I_1+I_2)+I_3.
\end{align}

\subsubsection{Estimate of $I_1$ and $I_2$}
\begin{align}
\label{intomega1}
    I_1+I_2=c_{1}(z_0)e^{-i\omega(z_0)}\int^{\frac{\pi}{4}}_{0}ie^{i\phi}\sin{2\phi} \,d\phi\int^{+\infty}_{0}U^2(a,y)f^{-2}(\rho e^{i\phi};z_0)(r(z_0+\rho \cos\phi)-r(z_0))d\rho.
\end{align}
We rewrite 
\begin{align}
\label{rewrite1}
    f^{-2}(\rho e^{i\phi};z_0)r(u)-r(z_0)&=(f^{-2}(\rho e^{i\phi};z_0)-1)(r(u)-r(z_0))\\
    \nonumber
    \quad &\quad  +(f^{-2}(\rho e^{i\phi};z_0)-1)r(z_0)+r(u)-r(z_0).
\end{align}
so \eqref{intomega1} can be further divided into three terms: $\widetilde{I}_0+\widetilde{I}_1+\widetilde{I}_2$ where
\begin{align}
\label{widetildeI0}
    &\widetilde{I}_0=c_{1}(z_0)e^{-i\omega(z_0)}\int^{\frac{\pi}{4}}_{0}ie^{i\phi}\sin{2\phi} d\phi\int^{+\infty}_{0}U^2(a,y)(f^{-2}(\rho e^{i\phi};z_0)-1)(r(u)-r(z_0))d\rho,\\
\label{widetildeI1}
    &\widetilde{I}_1=c_{1}(z_0)e^{-i\omega(z_0)}\int^{\frac{\pi}{4}}_{0}ie^{i\phi}\sin{2\phi} d\phi\int^{+\infty}_{0}U^2(a,y)(f^{-2}(\rho e^{i\phi};z_0)-1)r(z_0)d\rho,\\
\label{widetildeI2}
    &\widetilde{I}_2=c_{1}(z_0)e^{-i\omega(z_0)}\int^{\frac{\pi}{4}}_{0}ie^{i\phi}\sin{2\phi} d\phi\int^{+\infty}_{0}U^2(a,y)(r(u)-r(z_0))d\rho.
\end{align}
For $\widetilde{I}_0$ , according to \eqref{eqe} , we have that 
\begin{align}
    \widetilde{I}_0 = O\left(\frac{\ln t}{t^{\frac{5}{4}}}\right).
\end{align}
We then deal with $\widetilde{I}_2$. Integration by parts leads to
\begin{align}
    \widetilde{I}_2&=c_{1}(z_0)e^{-i\omega(z_0)}\int^{\frac{\pi}{4}}_{0}ie^{i\phi}\sin{2\phi} d\phi\int^{+\infty}_{0}e^{4it\rho^{2}e^{2i\phi}}A^{2}(a,y)(r(u)-r(z_0))d\rho \nonumber\\ 
    &=\frac{c_{1}(z_0)e^{-i\omega(z_0)}}{8t}\int^{\frac{\pi}{4}}_{0}e^{-i\phi}\sin{2\phi} d\phi\int^{+\infty}_{0}\frac{A^{2}(a,y)(r(u)-r(z_0))}{\rho}de^{4it\rho^{2}e^{2i\phi}}\nonumber\\
    &=\alpha_{1,1}t^{-1} +\widehat{I}_{2}-\overline{I}_2.\nonumber
\end{align}
where 
\begin{align}
    \alpha_{1,1}&=-\frac{A^{2}(a,0)r'(z_0)c_{1}(z_0)e^{-i\omega(z_0)}}{8}\int^{\frac{\pi}{4}}_{0}e^{-i\phi}\sin{2\phi}\cos\phi d\phi,\\
    \label{widehatI2}
    \widehat{I}_{2} &=\frac{c_{1}(z_0)e^{-i\omega(z_0)}}{8t}\int^{\frac{\pi}{4}}_{0}e^{-i\phi}\sin{2\phi} d\phi\int^{+\infty}_{0}e^{4it\rho^{2}e^{2i\phi}}\frac{A^{2}(a,y)(r(u)-r(z_0)-\rho r'(u)\cos\phi)}{\rho^{2}}d\rho,\\
    \label{overlineI2}
    \overline{I}_{2}& = \frac{c_{2,2}(z_0)}{t^{\frac{1}{2}}}\int^{\frac{\pi}{4}}_{0}\sin{2\phi} d\phi\int^{+\infty}_{0}e^{4it\rho^{2}e^{2i\phi}}\frac{B(a,y)(r(u)-r(z_0))}{\rho}d\rho.
    \end{align}
    For \eqref{widehatI2} , from \eqref{eqwidehatI2} , we have that
    \begin{align}
        \label{estimate11}
    &\widehat{I}_{2}  =O\left(t^{-\frac{5}{4}}\right).
 \end{align}
For \eqref{overlineI2} , recall that $\frac{dA(a,y)}{dy} = -(a+\frac{1}{2})A(a+1,y) $, so 
$$\frac{dA^{2}(a,y)}{d\rho} = -4\sqrt{2}e^{-\frac{\pi}{4}i}t^{\frac{1}{2}}e^{i\phi}(a+\frac{1}{2})A(a+1,y)A(a,y).$$ 
We then deduce  
\begin{align}
    \nonumber\overline{I}_{2} &= \frac{c_{2,2}(z_0)}{t^{\frac{1}{2}}}\int^{\frac{\pi}{4}}_{0}\sin{2\phi} d\phi\int^{+\infty}_{0}e^{4it\rho^{2}e^{2i\phi}}\frac{B(a,y)(r(u)-r(z_0))}{\rho}d\rho \\
    \nonumber
    &=\frac{c_{2,2}(z_0)}{t^{\frac{1}{2}}}\int^{\frac{\pi}{4}}_{0}\sin{2\phi} \cos\phi d\phi\int^{+\infty}_{0}e^{4it\rho^{2}e^{2i\phi}}B(a,y)r'(z_0)d\rho
    \\ \nonumber 
    &\quad +\frac{c_{2,2}(z_0)}{t^{\frac{1}{2}}}\int^{\frac{\pi}{4}}_{0}\sin{2\phi} d\phi\int^{+\infty}_{0}e^{4it\rho^{2}e^{2i\phi}}\frac{B(a,y)(r(u)-r(z_0)-\rho r'(z_0)\cos\phi)}{\rho}d\rho\\
    &:=\overline{I}_{2,1}+\overline{I}_{2,2}.
\end{align}
where 
\begin{align}
    c_{2,2}(z_0)&=-\frac{\sqrt{2}e^{-\frac{\pi}{4}i}(a+\frac{1}{2})c_{1}(z_0)e^{-i\omega(z_0)}}{2},\\
    \nonumber
    B(a,y) &= A(a+1,y)A(a,y).
    \end{align}
    Thus, replacing $\sqrt{t}\rho$ by $\rho$, for the first term we obtain
    \begin{align}
    \overline{I}_{2,1} &=  \frac{c_{2,2}(z_0)}{t}\int^{\frac{\pi}{4}}_{0}\sin{2\phi} \cos\phi d\phi\int^{+\infty}_{0}e^{4i\rho^{2}e^{2i\phi}}B(a,y)r'(z_0)d\rho\\
    \nonumber
    &:=-\alpha_{1,2}t^{-1}
    \end{align}
where
\begin{align}
    &\alpha_{1,2}= -c_{2,2}(z_0)r'(z_0)\int^{\frac{\pi}{4}}_{0}\sin{2\phi} \cos\phi d\phi\int^{+\infty}_{0}e^{4i\rho^{2}e^{2i\phi}}B(a,y)d\rho.
\end{align}
For the second term $\overline{I}_{2,2}$, by \eqref{eqoverlineI2} we have that
\begin{align}
\label{estimate12}
    \overline{I}_{2,2} &=  \frac{c_{2,2}(z_0)}{t^{\frac{1}{2}}}\int^{\frac{\pi}{4}}_{0}\sin{2\phi} d\phi\int^{+\infty}_{0}e^{4it\rho^{2}e^{2i\phi}}\frac{B(a,y)(r(u)-r(z_0)-\rho r'(z_0)\cos\phi)}{\rho}d\rho\\
    \nonumber
    &=O(t^{-\frac{5}{4}}).
\end{align}
%                       &= \frac{c_{2,2}(z_0)}{8it^{\frac{3}{2}}}\int^{\frac{\pi}{4}}_{0}e^{i\phi}\sin2\phi d\phi\int^{+\infty}_{0}\frac{B(a,y)(r(u)-r(z_0)-\rho r'(z_0)\cos\phi)}{\rho^{2}}de^{4it\rho^{2}e^{2i\phi}}\\
%    &=-\frac{c_{2,2}(z_0)r''(z_0)B(a,0)}{16it^{\frac{3}{2}}}\int^{\frac{\pi}{4}}_{0}e^{i\phi}\sin2\phi \cos\phi d\phi\\
%    &-\frac{c_{2,2}(z_0)}{8it^{\frac{3}{2}}}\int^{\frac{\pi}{4}}_{0}e^{i\phi}\sin2\phi d\phi\int^{+\infty}_{0}e^{4it\rho^{2}e^{2i\phi}}d\frac{B(a,y)(r(u)-r(z_0)-\rho r'(z_0)\cos\phi)}{\rho^{2}}

%\begin{align}
%    \widetilde{I}_{0}&=\frac{c_{1}(z_0)e^{-i\omega(z_0)}}{8t}\int^{\frac{\pi}{4}}_{0}e^{-i\phi}\sin2\phi d\phi\int^{+\infty}_{0}(f^{-2}(\rho e^{i\phi};z_0)-1)\frac{A^{2}(a,y)(r(u)-r(z_0))}{\rho}de^{4it\rho^{2}e^{2i\phi}}   \\
%     \nonumber
%    &=-\frac{c_{1}(z_0)e^{-i\omega(z_0)}}{8t}\int^{\frac{\pi}{4}}_{0}e^{-i\phi}\sin2\phi d\phi\int^{+\infty}_{0}e^{4it\rho^{2}e^{2i\phi}}d(f^{-2}(\rho e^{i\phi};z_0)-1)\frac{A^{2}(a,y)(r(u)-r(z_0))}{\rho} \triangleq \widehat{I}_{0}
%\end{align}
For $ \widetilde{I}_{1}$ , we substitute \eqref{functionf} into \eqref{widetildeI1}, then we have
\begin{align}
     \widetilde{I}_{1}&=-\alpha_{1,3}\frac{\ln t}{t}+\widehat{I}_1\\
\nonumber
& =-\alpha_{1,3}\frac{\ln t}{t}+ o\left(\frac{\ln t}{t} \right).
\end{align}
where
\begin{align}
\label{alpha13}
    &\alpha_{1,3}=-\frac{F'(z_0)r(z_0)c_{1}(z_0)e^{-i\omega(z_0)}}{2\pi }\int^{\frac{\pi}{4}}_{0}e^{i\phi}\sin{2\phi} d\phi\int^{+\infty}_{0}e^{4i\rho^{2}e^{2i\phi}}A^{2}(a,y)\rho e^{i\phi}d\rho.
\end{align}
\subsubsection{Estimates on $I_3$}
According to \eqref{boundf}, $f^{-2}$ is uniformly bounded, and also notice that $r\in H^{4,2}$, we have the following result followed by\eqref{esI3}
\begin{align}
    \nonumber I_3 = O(t^{-1}).
\end{align}

\subsection{Large $t$ expansion in $\Omega_4$}
In $\Omega_4$ we have that 
$$W_{12}(u,v;x,t)=-c_{4}(z_0)U^{2}(a,-y)\overline{\partial}E_{4}(u,v)$$ 
where 
\begin{align}
&A^{(-1)}_{2} = \frac{1}{\sqrt{2}}\left(1-|r(z_0)|^{2}\right)^{3/8}e^{\pi i/4}\exp\left(-i\frac{1}{4\pi}\ln2F(z_0)\right),\\
&c_{4}(z_0)=\beta^{2}(A^{(-1)}_{2})^{2}.
\end{align}
Then
\begin{align}
\label{Omega6}
    &\iint_{\Omega_4}W_{12}(u,v;x,t)dA(u,v)=-c_{4}(z_0)\int^{\frac{5\pi}{4}}_{\pi}d\phi\int_{0}^{+\infty}\rho U^{2}(a,-y)\overline{\partial}E_{4}d\rho\\\nonumber
    &=c_{4}(z_0)\int_{0}^{\infty}U^{2}(a,-y)d\rho\int^{\frac{5\pi}{4}}_{\pi}ie^{i\phi}\sin{2\phi}f^{-2}(u+iv;z_0)\frac{{r(u)}e^{-i\omega(z_0)}}{1-|r(u)|^{2}}-ie^{i\phi}\sin{2\phi}\frac{|r(z_0)|}{1-|r(z_0)|^{2}}\\\nonumber
    &-\frac{\rho}{2}\cos{2\phi}f^{-2}(u+iv;z_0)e^{-i\omega(z_0)}\frac{{r'(u)}+r^{2}(u)\overline{r'(u)}}{(1-|r(u)|^2)^2}d\phi \\
    \nonumber
    &:= \textrm{IV}_1+\textrm{IV}_2-\textrm{IV}_3.
\end{align}
If we donate $q_r(u) = \frac{r(u)}{1-|r(u)|^{2}}$, then the estimates over $\eqref{Omega6}$ are similar to that of $\eqref{Omega1}$. The following estimates can be obtained similarly:
\begin{align}
\nonumber
    &\textrm{IV}_1+\textrm{IV}_2 = (\alpha_{4,1}+\alpha_{4,2})t^{-1}+\alpha_{4,3}\frac{\ln t}{t}+\widetilde{\textrm{IV}_0}+\widehat{\textrm{IV}_2}+\overline{\textrm{IV}}_{2,2}+\widehat{\textrm{IV}_1},\\
    &|\textrm{IV}_3| =O(t^{-1}).
\end{align}
where
\begin{align}
   \nonumber &\alpha_{4,1}=-\frac{A^{2}(a,0)q_{r}'(z_0)c_{4}(z_0)e^{-i\omega(z_0)}}{8}\int^{\frac{5\pi}{4}}_{\pi}e^{-i\phi}\sin{2\phi}\cos\phi d\phi,\\ \nonumber
    &\alpha_{4,2} = -\frac{\sqrt{2}q_{r}'(z_0)e^{-\frac{\pi}{4}i}(a+\frac{1}{2})c_4(z_0)e^{-i\omega(z_0)}}{2}\int^{\frac{5\pi}{4}}_{\pi}\sin2\phi \cos\phi d\phi\int^{+\infty}_{0}e^{4i\rho^{2}e^{2i\phi}}B(a,-y)d\rho,\\  
    \label{alpha43}&\alpha_{4,3} = \frac{F'(z_0)q_{r}(z_0)c_{4}(z_0)e^{-i\omega(z_0)}}{2\pi }\int^{\frac{5\pi}{4}}_{\pi}e^{i\phi}\sin{2\phi} d\phi\int^{+\infty}_{0}e^{4i\rho^{2}e^{2i\phi}}A^{2}(a,-y)\rho e^{i\phi}d\rho.
\end{align}
According to \eqref{eqb}-\eqref{eqe}, we have the following results:
\begin{align}
    &\widehat{\mathrm{IV}_1} =  o\left(\frac{\ln t}{t}\right),   \\\nonumber
    &\widetilde{\mathrm{IV}_0} = c_{4}(z_0)e^{-i\omega(z_0)}\int^{\frac{5\pi}{4}}_{\pi}\sin{2\phi} d\phi\int^{+\infty}_{0}U^2(a,-y)(f^{-2}(\rho e^{i\phi};z_0)-1)(q_{r}(u)-q_{r}(z_0))d\rho\\\nonumber
    &\qquad=O\left(\frac{\ln t}{t^{\frac{5}{4}}}\right),\\
    &\widehat{\mathrm{IV}_2} = \frac{c_{4}(z_0)e^{-i\omega(z_0)}}{8t}\int^{\frac{5\pi}{4}}_{\pi}e^{-i\phi}\sin2\phi d\phi\int^{+\infty}_{0}e^{4it\rho^{2}e^{2i\phi}}\frac{A^{2}(a,-y)(q_{r}(u)-q_{r}(z_0)-\rho q_{r}'(u)\cos\phi)}{\rho^{2}}d\rho\\\nonumber
    &\qquad= O\left(t^{-\frac{5}{4}}\right),\\ 
    &\overline{\mathrm{IV}}_{2,2} =  -\frac{\sqrt{2}e^{-\frac{\pi}{4}i}(a+\frac{1}{2})c_{4}(z_0)e^{-i\omega(z_0)}}{2t^{\frac{1}{2}}}\int^{\frac{5\pi}{4}}_{\pi}\sin{2\phi} d\phi
    \int^{+\infty}_{0}e^{4it\rho^{2}e^{2i\phi}}\frac{B(a,-y)(r(u)-r(z_0)-\rho q_{r}'(z_0)\cos\phi)}{\rho}d\rho\\ \nonumber
    &\qquad= O\left(t^{-\frac{5}{4}}\right).
%    &\widetilde{\mathrm{IV}_3} =c_{4}(z_0)e^{-i\omega(z_0)}\int^{\frac{5\pi}{4}}_{\pi}\frac{1}{2}\cos{2\phi} d\phi \int^{+\infty}_{0}\rho e^{4it\rho^{2}e^{2i\phi}}f^{-2}(\rho e^{i\phi};z_0)A^{2}(a,-y)(q_{r}'(u)-q_{r}'(z_0))d\rho\\\nonumber
\end{align}
%let $\phi = \phi-\pi$,we have the fact that $\textrm{IV}_1+\textrm{IV}_2 $ %equals $\eqref{intomega1}$,only $\textrm{IV}_3$ is different. Similar to %$\eqref{widetilde{I}3}$,we rewrite $\frac{{r'(u)}+r^{2}%(u)\overline{r'(u)}}{(1-|r(u)|^2)^2}= \frac{{r'(u)}+r^{2}(u)\overline{r'(u)}}{(1-|r(u)|^2)^2}-\frac{{r'(z_0)}+r^{2}(z_0)\overline{r'(z_0)}}{(1-|r(z_0)|^2)^2}+\frac{{r'(z_0)}+r^{2}%%(z_0)\overline{r'(z_0)}}{(1-|r(z_0)|^2)^2}$,so that $\textrm{IV}_3$ = %$\widetilde{\textrm{IV}}_3$+$\widetilde{\textrm{IV}}_4$
%\begin{align}
%    \widetilde{\textrm{IV}}_4 &= -\frac{{r'(z_0)}+r^{2}(z_0)\overline{r'(z_0)}}{(1-|r(z_0)|^2)^2}\frac{c_{4}(z_0)e^{-i\omega(z_0)}}{8it}\int^{\frac{5\pi}{4}}_{\pi}\frac{\frac{1}{2}\cos2\phi}{e^{2i\phi}} d\phi\int^{+\infty}_{0}f^{-2}(u+iv;z_0)A^{2}(a,-y)de^{4i\rho^{2}e^{2i\phi}}\\ \nonumber
%    & = \frac{{r'(z_0)}+r^{2}(z_0)\overline{r'(z_0)}}{(1-|r(z_0)|^2)^2}\frac{A^{2}(a,0)c_{4}(z_0)e^{-i\omega(z_0)}}{8it}\int^{\frac{5\pi}{4}}_{\pi}\frac{\frac{1}{2}\cos2\phi}{e^{2i\phi}} d\phi+ \widehat{\textrm{IV}}_4
%\end{align}
\subsection{Long time asymptotics in $\Omega_6$}
In $\Omega_6$, recall that 
$$W_{12}=-c_{6}(z_0)U^{2}(-a,iy)\overline{\partial}E_{6}(u,v),$$
where 
\begin{align}
& B_{1}^{(0)}=\beta^{-1}\left(1-|r(z_0)|^{2}\right)^{-1/8}\exp(i\frac{1}{4\pi}\ln2F(z_0))\\
&c_6(z_0)=\beta^{2}(B_{1}^{(0)})^{2}.
\end{align}
 Then
\begin{align}
 \nonumber\iint_{\Omega_6}W_{12}dudv &= -c_6(z_0)\int_{-\frac{\pi}{4}}^{0}U^{2}(-a,iy)d\phi\int^{+\infty}_{0}-ie^{i\phi}\sin{2\phi}|r(z_0)|+ie^{i\phi}\sin{2\phi}f^{2}(u+iv;z_0)\overline{{r(u)}}e^{i\omega(z_0)}\\
 \nonumber
 & -\frac{\rho}{2}\cos{2\phi}f^{2}(u+iv;z_0)\overline{r'(u)}e^{i\omega(z_0)}d\rho\\
&:=\textrm{VI}_1+\textrm{VI}_2+\textrm{VI}_3+\textrm{VI}_4   
 \end{align}
where
\begin{equation}
\begin{aligned}
\label{integralin6}
    &\textrm{VI}_1 = -c_6(z_0)e^{i\omega(z_0)}\int^{0}_{-\frac{\pi}{4}}ie^{i\phi}\sin{2\phi}d\phi\int^{+\infty}_{0}e^{-4it\rho^{2}e^{2i\phi}}A^{2}(-a,iy)(\overline{r(u)}-\overline{r(z_0)})d\rho,\\
    &\textrm{VI}_2 =  -c_6(z_0)e^{i\omega(z_0)}\int^{0}_{-\frac{\pi}{4}}ie^{i\phi}\sin{2\phi}d\phi\int^{+\infty}_{0}e^{-4it\rho^{2}e^{2i\phi}}A^{2}(-a,iy)(f^{2}(\rho e^{i\phi};z_0)-1)(\overline{r(u)}-\overline{r(z_0)})d\rho,\\
    &\textrm{VI}_3 = -c_6(z_0)\overline{r(z_0)}e^{i\omega(z_0)}\int^{0}_{-\frac{\pi}{4}}ie^{i\phi}\sin{2\phi}d\phi\int^{+\infty}_{0}e^{-4it\rho^{2}e^{2i\phi}}A^{2}(-a,iy)(f^{2}(\rho e^{i\phi};z_0)-1)d\rho,\\  
    &\textrm{VI}_4 =c_6(z_0)e^{i\omega(z_0)}\int^{0}_{-\frac{\pi}{4}}\frac{1}{2}\cos{2\phi}d\phi\int^{+\infty}_{0} \rho e^{-4it\rho^{2}e^{2i\phi}}A^{2}(-a,iy)f^{2}(\rho e^{i\phi};z_0) \overline{r'(u)}d\rho.
\end{aligned}
\end{equation}
Similar to the estimates in $\Omega_{1}$, we have the following results:
\begin{align}
\nonumber
&\textrm{VI}_1=\alpha_{6,1}t^{-1}+\alpha_{6,2}t^{-1}+\widehat{\textrm{VI}_1}+\widetilde{\textrm{VI}_2},\\\nonumber
%    &\textrm{VI}_2 =  \frac{c_6(z_0)e^{i\omega(z_0)}}{-8t}\int^{0}_{-\frac{\pi}{4}}e^{-i\phi}\sin{2\phi}d\rho \int^{+\infty}_{0}A^{2}(-a,iy)(f^{2}(\rho e^{i\phi};z_0)-1)\frac{(\overline{r(u)}-\overline{r(z_0)})}{\rho}de^{-4it\rho^{2}e^{2i\phi}}= O(t^{-1}),\\\nonumber
    &\textrm{VI}_3 =\alpha_{6,3}\frac{\ln t}{t} +\widehat{\textrm{VI}_3},\\
    \label{VI4}
    &|\textrm{VI}_4|=O(t^{-1}) 
%    &\textrm{VI}_5 = \frac{c_6(z_0)e^{i\omega(z_0)}}{8it}\int^{0}_{-\frac{\pi}{4}}\frac{1}{2}\cos{2\phi}e^{-2i\phi}d\phi\int^{+\infty}_{0}A^{2}(-a,iy)f^{2}(\rho e^{i\phi};z_0) (\overline{r'(u)}-\overline{r'(z_0)})de^{-4it\rho^{2}e^{2i\phi}} = O(t^{-1}).
    \end{align}
where
\begin{align}
        \nonumber&\alpha_{6,1} = \frac{c_6(z_0)\overline{r'(z_0)}A^{2}(-a,0)e^{i\omega(z_0)}}{8}\int^{0}_{-\frac{\pi}{4}}e^{-i\phi}\sin{2\phi}\cos{\phi}d\phi,\\
        \nonumber&\alpha_{6,2} = -\frac{\sqrt{2}ie^{-\frac{\pi}{4}i}(-a+\frac{1}{2})\overline{r'}(z_0)c_{6}(z_0)e^{i\omega(z_0)}}{2}\int^{0}_{-\frac{\pi}{4}}\sin{2\phi}\cos{\phi}d\phi\int^{+\infty}_{0}e^{-4i\rho^2 e^{2i\phi}}B(-a,iy)d\rho,\\
        \label{alpha63}&\alpha_{6,3} =\frac{ c_6(z_0)\overline{r(z_0)}e^{i\omega(z_0)}}{2\pi }\int^{0}_{-\frac{\pi}{4}}e^{i\phi}\sin{2\phi}d\phi\int^{+\infty}_{0}e^{-4i\rho^2 e^{2i\phi}}A(-a,iy)F'(z_0)\rho e^{i\phi}d\rho.
\end{align}
And
\begin{align}
    \nonumber&\widehat{\textrm{VI}_1} = \frac{c_{6}(z_0)e^{i\omega(z_0)}}{8t}\int^{0}_{-\frac{\pi}{4}}e^{-i\phi}\sin{2\phi} d\phi\int^{+\infty}_{0}e^{4it\rho^{2}e^{2i\phi}}\frac{A^{2}(-a,iy)(\overline{r}(u)-\overline{r}(z_0)-\rho \overline{r'}(u)\cos\phi)}{\rho^{2}}d\rho,\\\nonumber
    &\widetilde{\textrm{VI}_{2}} =   -\frac{\sqrt{2}ie^{-\frac{\pi}{4}i}(-a+\frac{1}{2})c_{6}(z_0)e^{i\omega(z_0)}}{2t^{\frac{1}{2}}}\int^{0}_{-\frac{\pi}{4}}\sin{2\phi} d\phi\int^{+\infty}_{0}e^{4it\rho^{2}e^{2i\phi}}\frac{B(-a,iy)(\overline{r}(u)-\overline{r}(z_0)-\rho \overline{r'}(z_0)\cos\phi)}{\rho}d\rho.
%    &\widehat{\textrm{VI}_4} = -c_{6}(z_0)\overline{r'}(z_0)e^{i\omega(z_0)}\int^{0}_{-\frac{\pi}{4}}\frac{1}{2}\cos{2\phi} d\phi \int^{+\infty}_{0}\rho e^{4it\rho^{2}e^{2i\phi}}\left[f^{-2}(\rho e^{i\phi};z_0)-1\right]A^{2}(-a,i\rho)d\rho
\end{align}
 According to \eqref{eqb}-\eqref{eqe} and \eqref{functionf}, we have that 
\begin{align}
    \label{esVI1}
     &\widehat{\textrm{VI}_1} = O\left(t^{-\frac{5}{4}}\right),\\
     &\widetilde{\textrm{VI}_2} =O\left(t^{-\frac{5}{4}}\right),\\
     &\textrm{VI}_2 =O\left(\frac{\ln t}{t^{\frac{5}{4}}}\right),\\
     \label{esVI3}
     &\widehat{\textrm{V}}_3 = o\left(\frac{\ln t }{t}\right).
\end{align}
\subsection{Large $t$ expansion in $\Omega_3$}
In $\Omega_3$, 
$$W_{12} = c_{3}(z_{0})U^{2}(-a,-iy)\overline{\partial}E_{3}(u,v),$$
where 
\begin{align}
&B_{1}^{(1)}=\beta^{-1}\left(1-|r(z_0)|^{2}\right)^{3/8}\exp(i\frac{1}{4\pi}\ln2F(z_0)),\\
&c_3(z_0)=\beta^{2}(B_{1}^{(1)})^{2}.
\end{align}
 Then
\begin{align}
 \nonumber\iint_{\Omega_3}W_{12}dudv& = c_3(z_0)\int_{\frac{3\pi}{4}}^{\pi}U^{2}(-a,-iy)d\phi\\ \nonumber
 &\cdot \int^{+\infty}_{0}-ie^{i\phi}\sin{2\phi}\frac{|r(z_0)|}{1-|r(z_0)|^{2}}+ie^{i\phi}\sin{2\phi}f^{2}(u+iv;z_0)\frac{\overline{{r(u)}}e^{i\omega(z_0)}}{(1-\overline{{r(u)}^{2})^{2}}}\\\nonumber
 & \quad-\frac{\rho}{2}\cos{2\phi}f^{2}(u+iv;z_0)\frac{\overline{r'(u)}+(\overline{r^{2}(u)})r'(u)}{(1-|r(u)|^2)^2}d\rho\\ 
 &:=\textrm{III}_1+\textrm{III}_2+\textrm{III}_3+\textrm{III}_4.
\end{align}
Note that if we donate $\overline{q}_{r}(u) = \frac{\overline{r}(u)}{1-|\overline{r}(u)|^{2}}$, similar to $\eqref{integralin6}$, we have the following results:
\begin{align}
    &\textrm{III}_1 = c_3(z_0)e^{i\omega(z_0)}\int_{\frac{3\pi}{4}}^{\pi}ie^{i\phi}\sin{2\phi}d\phi\int^{+\infty}_{0}e^{-4it\rho^{2}e^{2i\phi}}A^{2}(-a,-iy)(\overline{q}_{r}(u)-\overline{q}_{r}(z_0)d\rho,\\\nonumber
    &\textrm{III}_2 =  c_3(z_0)e^{i\omega(z_0)}\int_{\frac{3\pi}{4}}^{\pi}ie^{i\phi}\sin{2\phi}d\phi\int^{+\infty}_{0}e^{-4it\rho^{2}e^{2i\phi}}A^{2}(-a,-iy)(f^{2}(\rho e^{i\phi};z_0)-1)(\overline{q}_{r}(u)-\overline{q}_{r}(z_0))d\rho,\\\nonumber
    &\textrm{III}_3 = c_3(z_0)\overline{q}_{r}(z_0)e^{i\omega(z_0)}\int_{\frac{3\pi}{4}}^{\pi}ie^{i\phi}\sin{2\phi}d\phi\int^{+\infty}_{0}e^{-4it\rho^{2}e^{2i\phi}}A^{2}(-a,-iy)(f^{2}(\rho e^{i\phi};z_0)-1))d\rho,\\  \nonumber
    &\textrm{III}_4 = -c_3(z_0)e^{i\omega(z_0)}\int_{\frac{3\pi}{4}}^{\pi}\frac{1}{2}\cos{2\phi}d\phi\int^{+\infty}_{0}\rho e^{-4it\rho^{2}e^{2i\phi}}A^{2}(-a,-iy)f^{2}(\rho e^{i\phi};z_0) \overline{q}^{'}_{r}(u)d\rho.
\end{align}
notice that 
\begin{align}  \nonumber&\textrm{III}_1=\alpha_{3,1}t^{-1}+\alpha_{3,2}t^{-1}+\widehat{\textrm{III}}_1+\widetilde{\textrm{III}_2},\\\nonumber
&\textrm{III}_3 =\alpha_{3,3}\frac{\ln t}{t} +\widehat{\textrm{III}}_3,\\
\label{III4}
&|\textrm{III}_4| =O(t^{-1})
\end{align}
where
    \begin{align}
        \nonumber&\alpha_{3,1} = \frac{c_3(z_0)\overline{q}_{r}(z_0)A^{2}(-a,0)e^{i\omega(z_0)}}{8}\int_{\frac{3\pi}{4}}^{\pi}e^{-i\phi}\sin{2\phi}\cos{\phi}d\phi,\\
         \nonumber&\alpha_{3,2} = \frac{\sqrt{2}ie^{-\frac{\pi}{4}i}\overline{q}^{'}_{r}(z_0)(-a+\frac{1}{2})c_{3}(z_0)e^{i\omega(z_0)}}{2}\int_{\frac{3\pi}{4}}^{\pi}\sin{2\phi}\cos{\phi}d\phi\int^{+\infty}_{0}e^{-4i\rho^2 e^{2i\phi}}B(-a,-iy)d\rho,\\
        \label{alpha33}&\alpha_{3,3} =-\frac{ c_3(z_0)\overline{q}_{r}(z_0)e^{i\omega(z_0)}}{2\pi }\int_{\frac{3\pi}{4}}^{\pi}e^{i\phi}\sin{2\phi}d\phi\int^{+\infty}_{0}e^{-4i\rho^2 e^{2i\phi}}A(-a,-iy)F'(z_0)\rho e^{i\phi}d\rho,
%        &\alpha_{3,4}=  -c_{3}(z_0)\overline{q}^{'}_{r}(z_0)e^{i\omega(z_0)}\int_{\frac{3\pi}{4}}^{\pi}
%\frac{1}{2}\cos{2\phi} d\phi \int^{+\infty}_{0}\rho e^{-4i\rho^{2}e^{2i\phi}}A^{2}(-a,-i\rho)d\rho
    \end{align}
and
\begin{align}
    &\widehat{\textrm{III}_1} = \frac{c_{3}(z_0)e^{i\omega(z_0)}}{8t}\int_{\frac{3\pi}{4}}^{\pi}e^{-i\phi}\sin2\phi d\phi\int^{+\infty}_{0}e^{4it\rho^{2}e^{2i\phi}}\frac{A^{2}(-a,-iy)((\overline{q}_{r}(u)-(\overline{q}_{r}(z_0)-\rho(\overline{q}_{r}(u)\cos\phi)}{\rho^{2}}d\rho,\\\nonumber
    &\widetilde{\textrm{III}_2} =   \frac{\sqrt{2}ie^{-\frac{\pi}{4}i}(-a+\frac{1}{2})c_{3}(z_0)e^{i\omega(z_0)}}{2t^{\frac{1}{2}}}\int_{\frac{3\pi}{4}}^{\pi}\sin{2\phi} d\phi\int^{+\infty}_{0}e^{4it\rho^{2}e^{2i\phi}}\frac{B(-a,-iy)(\overline{q}_{r}(u)-\overline{q}_{r}(z_0)-\rho \overline{q}^{'}_{r}(z_0)\cos\phi)}{\rho}d\rho.
%    &\widehat{\textrm{III}_4} = -c_{3}(z_0)r'(z_0)e^{i\omega(z_0)}\int_{\frac{3\pi}{4}}^{\pi}\frac{1}{2}\cos{2\phi} d\phi \int^{+\infty}_{0}\rho e^{4it\rho^{2}e^{2i\phi}}\left[f^{-2}(\rho e^{i\phi};z_0)-1\right]A^{2}(-a,i\rho)d\rho
\end{align}
Similarly to \eqref{esVI1}-\eqref{esVI3}, we have that 
\begin{align}
     &\widehat{\textrm{III}_1} = O\left(t^{-\frac{5}{4}}\right),\\
     &\widetilde{\textrm{III}_2} =O\left(t^{-\frac{5}{4}}\right),\\
     &\textrm{III}_2 = O\left(\frac{\ln t}{t^{\frac{5}{4}}}\right),\\
     &\widehat{\textrm{III}}_3 = o\left(\frac{\ln t }{t}\right).
\end{align}
\subsection{The higher order asymptotic formula}
\label{sub:higher}
We now put together the results from the previous subsections and derive the higher order asymptotic formula of $q(x,t)$. For $q_0(x) \in H^{2,4}(\bbR)$, $q(x,t)$ has the following asymptotic formula
\begin{align}
    q(x,t) = q^{(0)}(x,t)+\alpha_1 \frac{\ln t}{t}+o\left(\frac{\ln t}{t}\right).
\end{align}
where
\begin{align}
    q^{(0)}(x,t):=\mathrm{e}^{-\mathrm{i} \omega\left(z_0\right)} \mathrm{e}^{-2 \mathrm{i} t\theta\left(z_0 ; z_0\right)} c\left(z_0\right)^{-2}\left(2 t^{1 / 2}\right)^{-2 \mathrm{i} \nu\left(z_0\right)}\left[\frac{1}{2} t^{-1 / 2} \beta\left(\left|r\left(z_0\right)\right|\right)\right]
\end{align}
is the leading term given by \cite[(3)]{DMM}. For the higher order term we have that 
\begin{align}
\label{asy:alpha1}
    &\alpha_1 :=\frac{2i}{\pi}e^{i\omega(z_0)}e^{2it\theta(z_0;z_0)}c^{-2}(z_0)(2t^{\frac{1}{2}})^{-2i\nu (z_0)}\left(\sum_{i=1,3,4,6}(\alpha_{i,3})\right) = 0.
\end{align}
To see this, recall that $\alpha_{1,3}=\eqref{alpha13},\alpha_{4,3}=\eqref{alpha43},\alpha_{6,3}=\eqref{alpha63},\alpha_{3,3}=\eqref{alpha33}$. First notice that for $\alpha_{1,3}$ and $\alpha_{4,3}$, it is easy to check that 
    \begin{align}
          r(z_0)c_1(z_0)=q_r(z_0)c_4(z_0)
    \end{align}
    and replace $\phi$ with $\phi-\pi$, we have that
    \begin{align}
    \alpha_{4,3} &= \frac{F'(z_0)r(z_0)c_1(z_0)e^{-i\omega(z_0)}}{2\pi }\int^{\frac{5\pi}{4}}_{\pi}e^{i\phi}\sin{2\phi} d\phi\int^{+\infty}_{0}e^{4i\rho^{2}e^{2i\phi}}A^{2}(a,-y)\rho e^{i\phi}d\rho.\nonumber\\
    &=\frac{F'(z_0)r(z_0)c_1(z_0)e^{-i\omega(z_0)}}{2\pi }\int^{\frac{\pi}{4}}_{0}e^{i\phi}\sin{2\phi} d\phi\int^{+\infty}_{0}e^{4i\rho^{2}e^{2i\phi}}A^{2}(a,y)\rho e^{i\phi}d\rho = -\alpha_{1,3}.\label{eqq1}
    \end{align}
    Similarly for $\alpha_{3,3}$ and $\alpha_{6,3}$, we have that
    \begin{align}
    \label{eqq2}
        \alpha_{3,3} = -\alpha_{6,3}
    \end{align}
    So combining \eqref{eqq1} and \eqref{eqq2}, we have $\alpha_1=0$.

\appendix
\section{Estimation on certain integrals }
In this section, we investigate into the convergence of the following integrals in $\Omega_1$:
$\widehat{I}_2=\eqref{estimate11}$, $\overline{I}_{2,2}=\eqref{estimate12}$, $\widetilde{I}_{0}=\eqref{widetildeI0}$ and $I_3$ which is the third term of \eqref{Omega1} given by
\begin{align}
\label{I3}
    I_3 = c_{1}(z_0)\int_{0}^{+\infty}U^2(a;y)d\rho \int_{0}^{\frac{\pi}{4}}\frac{\rho}{2}\cos{2\phi} f^{-2}(\rho e^{i\phi};z_0)r'(z_0+\rho \cos\phi)e^{-i\omega(z_0)}d\phi.
\end{align}
\subsubsection{Estimation on$\widehat{I}_2,\overline{I}_{2,2},\widetilde{I}_{0}$ }
As a simple consequence of the fundamental theorem of calculus and \textit{Cauchy-Schwarz} inequality, we have:
\begin{align}
\label{inequality2}
    \left|r(u)-r(z_0)\right| &\leq \int^{u}_{z_0}\left|r'(w)\right|dw\\ 
    \nonumber
    &\leq \Vert r'\Vert_{L^{2}}\left[ (u-z_0)^{2}+v^{2}\right]^{\frac{1}{2}}.
\end{align}
\begin{align}
\label{inequaity1}
    \left|r(u)-r(z_0)-(u-z_0)r'(v)\right|&\leq \int^{u}_{z_0}\left|r'(w)-r'(v)\right|dw\\
    \nonumber
    &\leq  \int^{u}_{z_0} \int^{w}_{v}\left|r''(z)\right|dzdw\\ \nonumber
    &\leq \int^{u}_{z_0}\left(\int_{v}^{w}|dz|\right)^{1/2}\left(\int_{v}^{w}|r''(z)|^{2}|dz|\right)^{1/2}dw\\
    \nonumber
    &\leq \Vert r''\Vert_{L^{2}_\mathbb{R}}\int^{u}_{z_0}\sqrt{|w-v|}dw.
\end{align}
According to $\eqref{inequaity1}$, we have

\begin{align}
\label{inequality3}
    &\left|(r(u)-r(z_0)-\rho r'(u)\cos\phi)\right| \lesssim \int^{u}_{z_0}\sqrt{|w-u|}dw \lesssim (u-z_0)^{\frac{3}{2}} \lesssim \left[ (u-z_0)^{2}+v^{2}\right]^{\frac{3}{4}} = \rho^{\frac{3}{2}},\\
    \nonumber
    &\left|(r(u)-r(z_0)-\rho r'(z_0)\cos\phi)\right| \lesssim \int^{u}_{z_0}\sqrt{|w-z_0|}dw \lesssim (u-z_0)^{\frac{3}{2}} \lesssim \left[ (u-z_0)^{2}+v^{2}\right]^{\frac{3}{4}} = \rho^{\frac{3}{2}}.
\end{align}
We now turn to the estimate of \eqref{estimate11} and \eqref{estimate12}. According to \eqref{inequality3}, we obtain:
\begin{align}
\label{eqb}
   \left|\int^{+\infty}_{1}\sin{2\phi}e^{4it\rho^{2}e^{2i\phi}}\frac{A^{2}(a,y)(r(u)-r(z_0)-\rho r'(u)\cos\phi)}{\rho^{2}}d\rho \right| &\lesssim \int^{+\infty}_{1}\sin{2\phi}e^{-4t\rho^{2}\sin{2\phi}}\frac{1}{\rho^{\frac{1}{2}}}d\rho\\
   \nonumber
    &= O(t^{-\frac{1}{4}}),\\
   \left|\int^{+\infty}_{0}\sin{2\phi}e^{4it\rho^{2}e^{2i\phi}}\frac{B(a,y)(r(u)-r(z_0)-\rho r'(z_0)\cos\phi)}{\rho}d\rho  \right| &\lesssim \int^{+\infty}_{0} \sin{2\phi}e^{-4t\rho^{2}\sin{2\phi}}\rho^{\frac{1}{2}}d\rho\\
   \nonumber
   &= O(t^{-\frac{3}{4}}).\\
%   \left|\int^{+\infty}_{0}\rho e^{4it\rho^{2}e^{2i\phi}}f^{-2}(\rho e^{i\phi};z_0)A^{2}(a,y)(r'(u)-r'(z_0))d\rho\right|& \lesssim \int_{0}^{+\infty}\rho^{\frac{3}{2}}e^{-4t\rho^{2}\sin{2\phi}}d\rho\\
   \nonumber
%   &= O(t^{-\frac{5}{4}})
\end{align}
So one deduces that
\begin{align}
\label{eqwidehatI2}
    &\left| \widehat{I}_{2}\right| \lesssim O(t^{-1}t^{-\frac{1}{4}}) = O(t^{-\frac{5}{4}}),
\end{align}
\begin{align}
\label{eqoverlineI2}
    & \left| \overline{I}_{2,2} \right| \lesssim O(t^{-\frac{1}{2}}t^{-\frac{3}{4}})= O(t^{-\frac{5}{4}}).
%    &\left| \widetilde{I}_3 \right|= O(t^{-\frac{5}{4}})
\end{align}
For $\widetilde{I}_{0}$, we use $\eqref{functionf}$ and $\eqref{inequality2}$  to deduce
\begin{align}
\label{eqe}
   \left|\int^{+\infty}_{0}\sin{2\phi}e^{4it\rho^{2}e^{2i\phi}}(f^{-2}(\rho e^{i\phi};z_0)-1)(r(u)-r(z_0))d\rho\right| 
   & \lesssim \int^{+\infty}_{0}\sin{2\phi}e^{-4t\rho^{2}\sin{2\phi}}|f^{-2}(\rho e^{i\phi};z_0)-1|\rho^{\frac{1}{2}}d\rho\\
    \nonumber
    &\lesssim  \left(\frac{1}{t^\frac{5}{4}}+\frac{\ln t}{t^{\frac{5}{4}}}\right)\int^{+\infty}_{0}\sin{2\phi}e^{-4\rho^{2}\sin{2\phi}}\rho^{\frac{1}{2}}d\rho+o\left(\frac{lnt}{t^{\frac{5}{4}}}\right)\\
    \nonumber
    &  =O\left(\frac{\ln t}{t^{\frac{5}{4}}}\right).
\end{align}
%\begin{align}
%    \left|\int^{+\infty}_{0}\rho e^{4it\rho^{2}e^{2i\phi}}\left[f^{-2}(\rho e^{i\phi};z_0)-1\right]A^{2}(a,y)d\rho\right|&\lesssim \int^{+\infty}_{0}e^{-4\rho^{2}\sin{2\phi}}\rho|f^{-2}(\rho e^{i\phi};z_0)-1|d\rho \\
%    \nonumber
%    &\lesssim \left(\frac{1}{t^\frac{3}{2}}+\frac{\ln t}{t^{\frac{3}{2}}}\right)\int^{+\infty}_{0}e^{-4t\rho^{2}\sin{2\phi}}\rho d\rho+o\left(\frac{lnt}{t^{\frac{3}{2}}}\right).
%\end{align}
%So we conclude that
%\begin{align}
%    \widehat{I}_{4} = O\left(\frac{\ln t}{t^{\frac{3}{2}}}\right)
%\end{align}
\subsubsection{Estimation on $I_3$}
We rewrite $$f^{-2}(\rho e^{i\phi};z_0)r'(z_0+\rho \cos\phi) =\left(f^{-2}(\rho e^{i\phi};z_0)-1\right)r'(z_0+\rho \cos\phi)+r'(z_0+\rho \cos\phi).$$ So \eqref{I3} can be divided into two terms: $\widehat{I}_3+\overline{I}_3$ with
\begin{align}
\label{whI3}
    &\widehat{I}_3 = c_{1}(z_0)\int_{0}^{+\infty}U^2(a;y)d\rho \int_{0}^{\frac{\pi}{4}}\frac{\rho}{2}\cos{2\phi} \left(f^{-2}(\rho e^{i\phi};z_0)-1\right)r'(z_0+\rho \cos\phi)e^{-i\omega(z_0)}d\phi,\\
\label{olI3}
    &\overline{I}_3=c_{1}(z_0)\int_{0}^{+\infty}U^2(a;y)d\rho \int_{0}^{\frac{\pi}{4}}\frac{\rho}{2}\cos{2\phi}r'(z_0+\rho \cos\phi)e^{-i\omega(z_0)}d\phi.
\end{align}
According to \eqref{functionf} , \eqref{whI3} has higher decay than \eqref{olI3}.so we first focus on \eqref{olI3}.
\begin{align}
\overline{I}_3 &= c_{1}(z_0)\int_{0}^{+\infty}U^2(a;y)d\rho \int_{0}^{\frac{\pi}{4}}\frac{\rho}{2}\cos{2\phi} r'(z_0+\rho \cos\phi)e^{-i\omega(z_0)}d\phi\nonumber\\
& = \frac{c_{1}(z_0)e^{-i\omega(z_0)}}{2}\int_{0}^{\frac{\pi}{4}}\cos{2\phi}d\phi\int_{0}^{+\infty}\rho e^{4it\rho^{2}e^{2i\phi}}A^{2}(a,y)r'(z_0+\rho \cos\phi)d\rho\nonumber\\
&=\frac{c_{1}(z_0)e^{-i\omega(z_0)}}{2t}\int_{0}^{\frac{\pi}{4}}\frac{\cos{2\phi}}{8ie^{2i\phi}}d\phi\int_{0}^{+\infty}A^{2}(a,y)r'(z_0+\rho \cos\phi)de^{4it\rho^{2}e^{2i\phi}}\nonumber\\
&=-\frac{c_{1}(z_0)e^{-i\omega(z_0)}}{2t}\int_{0}^{\frac{\pi}{4}}\frac{\cos{2\phi}}{8ie^{2i\phi}}A^{2}(a,y)r'(z_0)d\phi-\overline{I}_{3,1}-\overline{I}_{3,2}
\end{align}
where
\begin{align}
&\overline{I}_{3,1}= \frac{c_{1}(z_0)e^{-i\omega(z_0)}}{2t}\int_{0}^{\frac{\pi}{4}}\frac{\cos{2\phi}\cos{\phi}}{8ie^{2i\phi}}d\phi\int_{0}^{+\infty}A^{2}(a,y)e^{4it\rho^{2}e^{2i\phi}}r''(z_0+\rho \cos\phi)d\rho,\\
&\overline{I}_{3,2} =  \frac{c_{2,2}(z_0)e^{-i\omega(z_0)}}{4t^{\frac{1}{2}}}\int_{0}^{\frac{\pi}{4}}\frac{\cos{2\phi}}{ie^{i\phi}}d\phi\int_{0}^{+\infty}e^{4it\rho^{2}e^{2i\phi}}r'(z_0+\rho \cos\phi)B(a,y)d\rho.
\end{align}
Since
\begin{align}
    &\left| \int_{0}^{+\infty}A^{2}(a,y)e^{4it\rho^{2}e^{2i\phi}}r''(z_0+\rho \cos\phi)d\rho\right| \lesssim \int_{0}^{+\infty}e^{-4t\rho^{2}\sin{2\phi}}d\rho = O(t^{-\frac{1}{2}}),\\
    &\left|\int_{0}^{+\infty}e^{4it\rho^{2}e^{2i\phi}}r'(z_0+\rho \cos\phi)B(a,y)d\rho\right| \lesssim \int_{0}^{+\infty}e^{-4t\rho^{2}\sin{2\phi}}d\rho = O(t^{-\frac{1}{2}})
\end{align}
we arrive at the following results:
\begin{align}
\label{overlineI31}
    \overline{I}_{3,1} = O(t^{-\frac{3}{2}}),\\
\label{overlineI32}
    \overline{I}_{3,2} = O(t^{-1}).
\end{align}
Combining \eqref{overlineI31} and \eqref{overlineI32}, we have:
\begin{align}
\label{overlineI3}
\overline{I}_3 = O(t^{-1})
\end{align}
which leads to 
\begin{align}
\label{esI3}
    I_3 = O(t^{-1}).
\end{align}
\section{The linear Schr\"odinger equation}
In this appendix we obtain the higher order long-time asymptotics of the linear Schr\"odinger equation  
\begin{equation}
\label{eq:LS}
{i}q_t+q_{xx}=0
\end{equation}
 with initial data decaying for large $x$:
\begin{equation}
    q_0 =q(x, 0)\in H^{2m,2m}(\mathbb{R})
\end{equation}
where $H^{2m,2m}(\mathbb{R})$ is the weighted Sobolev space defined by
\begin{equation}
    H^{m,n}(\bbR):=\{h: \left \langle1+|\xi|^2\right\rangle^{\frac{m}{2}}\widehat{h}(\xi)\in L^2(\bbR),\left \langle1+|x|^2\right \rangle^{\frac{n}{2}}h(x)\in L^2(\bbR)\}.
\end{equation}
The asymptotic formula for $q(x,t)$
takes the following form as $t \rightarrow\infty$:
\begin{align}
    q(x,t)&=t^{-1/2}\frac{\widehat{q}_0(z_0)e^{-i\pi/4}}{2\sqrt{\pi}}e^{ix^2/4t}+\sum_{k=1}^{n}t^{-k}\widehat{q}^{(2k-1)}_{0}(z_0)e^{4itz_{0}^2}2i\beta_{k}\\
    &+\sum_{k=1}^{n}t^{-(2k+1)/2}\widehat{q}^{(2k)}_{0}(z_0)e^{4itz_{0}^2}2i\gamma_{k}+\varepsilon(x,t)
\end{align}
where $\varepsilon(x,t)$ is an error term and $\beta_{k},\gamma_{k}$ is defined by
\begin{align}
    &\beta_{k} = \int_{-\frac{\pi}{4}}^{0}\frac{\frac{1}{2}\cos{2\phi}\cos^{2k-2}\phi}{(2k-3)!!(8ie^{2i\phi})^k}-\frac{ie^{i\phi}\sin{2\phi} \cos^{2k}\phi}{(2k)!!(8ie^{2i\phi})^k}d\phi,\\
    &\gamma_{k} = \int_{-\frac{\pi}{4}}^{0}\frac{\frac{1}{2}(\frac{\pi}{i})^{1/2}\cos{2\phi}\cos^{2k-2}(\phi)}{2{\sqrt{\cos{2\phi}}(8ie^{2i\phi})^k}}-\frac{ie^{i\phi}\sin{2\phi}\cos^{2k-1}(\phi)(\frac{\pi}{i})^{1/2}}{2(2k-1)!!\sqrt{\cos{2\phi}}(8ie^{2i\phi})^k}d\phi.
\end{align}
\begin{problem}
Given parameters (x,t) $\in \mathbb{R}^2$,find a 2x2 matrix $\boldmath{N}=\boldmath{N}(z)=\boldmath{N}(z;x,t)$ satisfying  the following conditions:\\
\textbf{Analyticity} $\boldmath{N}$ is analytic in the domain $\mathbb{C}\backslash \mathbb{R}$ and it has a continuous extension to the real axis from the upper (lower) half-plane denoted $N_{+}(z)(N_{-}(z))$ for $z \in \mathbb{R}$\\
\textbf{Jump Condition}The boundary values satisfy the jump condition
\begin{equation}
    N_{+}(z)=N_{-}(z)V_{N}(z)
\end{equation}
where the jump matrix $V_{N}(z)$ is defined by
\begin{equation*}V_{N}(z) := \begin{pmatrix}
1 & - \widehat{q}(z)e^{-2i\theta(z;z_{0})} \\
0 & 1
\end{pmatrix},z\in \mathbb{R},\theta(z;z_{0})=2z-4zz_{0},z_{0}=-\frac{x}{4t}.
\end{equation*}
\end{problem}
\subsection{Normalization} There is a matrix $N_{1}(x,t)$ such that
\begin{equation}
    N(z) = I +z^{-1}N_{1}(x,t)+o(z^{-1}),z\rightarrow \infty.
\end{equation}
one defines a function $q(x,t),(x,t) \in \mathbb{R}$ by
\begin{equation}
    q(x,t) :=2iN_{1,12}(x,t). 
\end{equation}
Then $q(x,t)$ is the  the solution of the Cauchy
problem.
By Fourier transform theory, if
\begin{equation}
  \widehat{q}_{0}(\xi)=\int_{\mathbb{R}}q_{0}(x)e^{2ix\xi}dx,\xi \in \mathbb{R} 
\end{equation}
is the Fourier transform of the initial data $q_0$,then $\widehat{q}_{0}$ lies in the weighted Sobolev space $H^{2m,2m}(\mathbb{R})$ as a function of $\xi \in \mathbb{R}$.The solution of the Cauchy problem is given in terms of $\widehat{q}(\xi)$ by the integral
\begin{equation}
    q(x,t)=\frac{1}{\pi}\int_{\mathbb{R}}\widehat{q}_{0}(z)e^{-2it\theta(z,z_{0})}dz.
\end{equation}
 It is worth noticing that this formula is exactly the formula in which case the solution of the Riemann-Hilbert Problem is explicitly given by
\begin{equation}N(z;x,t) :=\mathbb{I}-\frac{1}{2\pi i } \int_{\mathbb{R}}\frac{\widehat{q}(z)e^{-2i\theta(z;z_{0})}}{\xi -z}d\xi\begin{pmatrix}
0 & 1 \\
0 & 0
\end{pmatrix}.\end{equation}
\begin{theorem}{ Stokes’ theorem}
    let $\Omega$ denote a simply-connected region in the complex plane
with counter-clockwise oriented piecewise-smooth boundary$\partial \Omega$. If $f:\Omega \rightarrow \mathbb{C}$ is
differentiable as as function of two variables($u=Re(z)$and $v=Im(z)$) and
and extends continuously to $\partial \Omega$ then 
\begin{equation}
\label{stokes}
    \oint_{\partial\Omega}f(u,v)dz = \iint_{\Omega}2i\overline{\partial}f(u,v)dA(u,v)
\end{equation}
where $dA(u,v)$denotes area measure in the plane.
\end{theorem}
We define a function $E(u,v)$ on $\Omega_{+}\cup \Omega{-}$ as follows:
\begin{align}
E(u,v)=\cos{\left(2\arg(u+iv-z_0)\right)}\widehat{q}_{0}(u)
+(1-\cos{\left(2\arg(u+iv-z_0)\right)}\widehat{q}_{0}(z_0)\\
u+iv \in \Omega_{+}\cup \Omega{-}. \nonumber
\end{align}
It follows:\\
$E(u,0)=\widehat{q}_{0}(u)$ on the boundary $v=0$ since $\cos{\left(2\arg(u+iv-z_0)\right)}\equiv 1$,\\
$E(u,v-z_0)=\widehat{q}_{0}(z_0)$ on the boundary $v=z_0-u$ since $\cos{(2\arg(u+iv-z_0))}\equiv 0$.
\begin{theorem}
    If $\widehat{q}_{0}(z) \in H^{2m,2m}(\mathbb{R})$,for given $E(u,v)$ in 1.13,we have:
    \begin{equation}
    \label{linearcase}
        q(x,t) = \frac{1}{\pi}\int_{z_0+\infty e^{3\pi i/4}}^{z_0+\infty e^{-\pi i/4}}\widehat{q}_{0}(z_0)e^{-2i\theta(z;z_{0})}dz\\
        + \frac{1}{\pi}\iint_{\Omega_{+}- \Omega{-}}2i\overline{\partial}(E(u,v)e^{-2i\theta(z;z_{0})})dA(u,v).
    \end{equation}
\end{theorem}
\begin{proof}
    apply Stokes’ theorem to the functions $\pm E(u,v)e^{-2i\theta(z;z_{0})}$ on domains $\Omega_{\pm}$\\
    \begin{align*}
    \iint_{\Omega_{+}}2i\overline{\partial}(E(u,v)e^{-2i\theta(z;z_{0})})dA(u,v)&=\oint_{\partial{\Omega_{+}}}E(u,v)e^{-2i\theta(z;z_{0})}dz\\
    &=\left(\int_{-\infty}^{0}+\int_{0}^{+\infty e^{3\pi i/4}}\right)E(u,v)e^{-2i\theta(z;z_{0})}dz
    \end{align*}
    and similarly
    \begin{align*}
    \iint_{\Omega_{-}}2i\overline{\partial}(-E(u,v)e^{-2i\theta(z;z_{0})})dA(u,v)&=\left(\int_{0}^{+\infty}-\int_{0}^{+\infty e^{-\pi i/4}}\right)E(u,v)e^{-2i\theta(z;z_{0})}dz.
    \end{align*}
    add up the results and apply the boundary value of $E(u,v)$ and \eqref{stokes}, we obtain \eqref{linearcase}.
\end{proof}
It is easy to check that the first term can be explicitly calculated,that is 
\begin{equation}
    \frac{1}{\pi}\int_{z_0+\infty e^{3\pi i/4}}^{z_0+\infty e^{-\pi i/4}}\widehat{q}_{0}(z_0)e^{-2i\theta(z;z_{0})}dz = t^{-1/2}\frac{\widehat{q}_{0}e^{-\pi /4}}{2\sqrt{\pi}}e^{ix^{2}/(4t)}
\end{equation}
Next, we continue to calculate the second term of $q(x,t)$ through the method of stationary phase  to deduce an asymptotic expansion of $q(x,t)$
 First note that $\overline{\partial}e^{-2i\theta(z;z_{0})}\equiv 0 $ since $e^{-2i\theta(z;z_{0})}$ is an entire function. A direct computation gives
 \begin{align}
     \overline{\partial}(E(u,v)e^{-2i\theta(z;z_{0})})& =  \overline{\partial}E(u,v)e^{-2i\theta(z;z_{0})}\\ \nonumber&= \frac{1}{2}\cos{\left(2\arg(u+iv-z_0)\right)}\widehat{q}_{0}'(u)e^{-2i\theta(z;z_{0})}\\ \nonumber
     &+({q}_{0}(u)-{q}_{0}(z_0))\overline{\partial}\cos{\left(2\arg(u+iv-z_0)\right)}.
 \end{align}
 In polar coordinates $(\rho,\phi)$ centered at the point $z_0 \in \mathbb{R}$ and defined by 
 $$u=z_0+\rho \cos{\phi}, \quad v=\rho \sin{\phi},$$ 
 the \textit{Cauchy-Riemann} operator $\overline{\partial}$ takes the equivalent form
\begin{align}
\overline{\partial} :=\frac{1}{2}(\frac{\partial}{\partial u}+i\frac{\partial}{\partial v}) = \frac{e^{i\phi}}{2}(\frac{\partial}{\partial \rho}+\frac{i}{\rho}\frac{\partial}{\partial \phi}).
\end{align}
so as $\arg(u+iv-z_0)=\phi$,we have
\begin{equation}
    \overline{\partial}\cos{\left(2\arg(u+iv-z_0)\right)}=\frac{ie^{i\phi}}{2\rho}\frac{d}{d\phi}\cos{2\phi}=-\frac{ie^{i\phi}}{\rho}\sin{2\phi}.
\end{equation}
Combine formula 1.16 and 1.17,we obtain
\begin{equation}
    \overline{\partial}E(u,v) = \frac{1}{2}\cos{2\phi}\widehat{q}'_0(z_0+\rho \cos{\phi})-\frac{ie^{i\phi}}{\rho}\sin{2\phi}({q}_{0}(z_0+\rho \cos{\phi})-{q}_{0}(z_0)).
\end{equation}
Then
\begin{align}
    \int_{\Omega_{+}} \overline{\partial} E(u, v) e^{-2 i t \theta\left(z, z_{0}\right)} dA(u,v) = \int_{-\frac{\pi}{4}}^{0} d \phi \int_{0}^{+\infty} \frac{1}{2} \cos{2\phi} \widehat{q}^{\prime}\left(z_{0}+\rho \cos{\phi}\right) \rho e^{-2 i t\left[2 \rho^{2} e^{2 i \phi}-2 z_{0}^{2}\right]} d\rho
		\\-i\int_{-\frac{\pi}{4}}^{0} d \phi \int_{0}^{+\infty} e^{i \phi} \sin{2\phi}\left[\widehat{q}_{0}\left(z_{0}+\rho \cos{\phi}\right)-\widehat{q}_{0}\left(z_{0}\right)\right] e^{-2 i t\left[2 \rho^{2} e^{2 i \phi}-2 z_{0}^{2}\right]} d \rho \triangleq \mathbb{I}_1-i\mathbb{I}_2
\end{align}
We first deal with $\mathbb{I}_1$ ,donate $t = t(\rho,\phi) =z_{0}+\rho\cos{\phi}$
\begin{flalign}
    &\ \mathbb{I}_1=\int_{-\frac{\pi }{4} }^{0}\frac{1}{2}\cos{2\phi}e^{4itz_{0}^{2}}d\phi\int_{0}^{+\infty}\widehat{q}'_{0}\left ( t \right ) \rho e^{-4it\rho^{2}e^{2i\phi}}d\rho \nonumber\\
    &\ = \int_{-\frac{\pi}{4}}^{0} \frac{\frac{1}{2} \cos{2\phi} e^{4 i t z_{0}^{2}}}{\left(-8i t e^{2 i \phi}\right)} d\phi \int_{0}^{+\infty} \widehat{q}_{0}^{\prime}\left(t\right) d e^{-4 i t\rho^{2}e^{2i\phi}} \nonumber\\
    &\ = -\int_{-\frac{\pi}{4}}^{0} \frac{\frac{1}{2}\cos{2\phi}e^{4 i t z_{0}^{2}}}{\left(-8i t e^{2 i \phi}\right)} \widehat{q}_{0}^{\prime}\left(z_{0}\right) d\phi-\int_{-\frac{\pi}{4}}^{0} \frac{\frac{1}{2} \cos{2\phi} e^{4 i t z_{0}^{2}}}{\left(-8 i t e^{2 i \phi}\right)} d \phi
		\int_{0}^{+\infty} {\widehat{q} }_{0}'' \left ( t \right ) \cos{\phi}e^{-4ite^{2i\phi}\rho^{2}}d\rho \nonumber\\
    &= -\int_{-\frac{\pi}{4}}^{0} \frac{\frac{1}{2} \cos{2\phi}e^{4 i t z_{0}^{2}}}{\left(-8i t e^{2 i \phi}\right)} \widehat{q}_{0}^{\prime}\left(z_{0}\right) d\phi-\int_{-\frac{\pi}{4}}^{0} \frac{\frac{1}{2} \cos{2\phi}e^{4 i t z_{0}^{2}}\cos{\phi}}{\left(-8i t e^{2 i \phi}\right)}d\phi\int_{0}^{+\infty} {\widehat{q}_{0} }''\left ( z_{0} \right ) e^{-4ite^{2i\phi}\rho^{2}}d\rho \nonumber\\
	& -\int_{-\frac{\pi}{4}}^{0} \frac{\frac{1}{2} \cos{2\phi}e^{4 i t z_{0}^{2}}\cos{\phi}}{\left(-8i t e^{2 i \phi}\right)^2} d\phi\int_{0}^{+\infty} \frac{{\widehat{q}_{0}}'' \left ( t \right ) -{\widehat{q}_{0} }'' \left ( z_{0} \right ) }{\rho}de^{-4ite^{2i\phi}\rho^{2}} \nonumber\\
    &=\int_{-\frac{\pi}{4}}^{0} \frac{\frac{1}{2} \cos{2\phi} e^{4 i t z_{0}^{2}}}{\left(8i t e^{2 i \phi}\right)}  \widehat{q}_{0}^{\prime}\left(z_{0}\right) \mathrm{d}\phi
	+\int_{-\frac{\pi}{4}}^{0} \frac{\frac{1}{2} \cos{2\phi} e^{4 i t z_{0}^{2}}}{\left(8 i t e^{2 i \phi}\right)}d\phi\int_{0}^{+\infty} {\widehat{q} }''\left ( z_{0} \right ) \cos{\phi} e^{-4ite^{2i\phi}\rho^{2}}\mathrm{d}\rho\\
	&+\int_{-\frac{\pi}{4}}^{0} \frac{\frac{1}{2} \cos{2\phi} e^{4 i t z_{0}^{2}}\cos^{2}(\phi)}{\left(8 i t e^{2 i \phi}\right)^{2}}\widehat{q}_{0}^{(3)}(z_0)d\phi\\
	&+\int_{-\frac{\pi}{4}}^{0} \frac{\frac{1}{2} \cos{2\phi} e^{4 i t z_{0}^{2}}\cos{\phi}}{\left(8 i t e^{2 i \phi}\right)^{2}}d\phi
	\int_{0}^{+\infty} e^{-4ite^{2i\phi}\rho^{2}}\left [ \frac{\widehat{q}_{0}^{(3)}\left ( t \right )\rho\cos{\phi}  - \widehat{q}_{0}'' \left ( t \right ) -{\widehat{q}_{0}}'' \left ( z_{0} \right ) }{\rho^2}\right ] \mathrm{d}\rho \triangleq \mathbb{I}_{3}.
\end{flalign}
\begin{flalign}
    &\mathbb{I}_{3} = \frac{1}{2}\int_{-\frac{\pi}{4}}^{0} \frac{\frac{1}{2} \cos(2\phi)) e^{4 i t z_{0}^{2}}\cos^{3}(\phi)}{\left(8 i t e^{2 i \phi}\right)^{2}}\widehat{q}_{0}^{(4)}(z_0)d\phi\int_{0}^{+\infty} e^{-4ite^{2i\phi}\rho^{2}}d\rho\\
    &+\int_{-\frac{\pi}{4}}^{0} \frac{\frac{1}{2} \cos{2\phi} e^{4 i t z_{0}^{2}}\cos{\phi}}{\left(-8 i t e^{2 i \phi}\right)^{3}}d\phi\int_{0}^{+\infty} \frac{\widehat{q}_{0}^{(3)}(t)\rho\cos{\phi}-\widehat{q}_{0}^{\prime \prime}(t)+\widehat{q}_{0}^{\prime \prime}\left(z_{0}\right) -\frac{1}{2}\widehat{q}_{0}^{(4)}(z_0) \rho^{2} \cos ^{2}(\phi)}{\rho^{3}} d e^{-4 i t e^{2 i \phi} \rho^{2}}\\
    &= \frac{1}{2}\int_{-\frac{\pi}{4}}^{0} \frac{\frac{1}{2} \cos{2\phi} e^{4 i t z_{0}^{2}}\cos^{3}(\phi)}{\left(8 i t e^{2 i \phi}\right)^{2}}\widehat{q}_{0}^{(4)}(z_0)d\phi\int_{0}^{+\infty} e^{-4ite^{2i\phi}\rho^{2}}d\rho+\frac{1}{3} \int_{-\frac{\pi}{4}}^{0} \frac{\frac{1}{2} \cos{2\phi} e^{4 i t z_{0}^{2}}\cos^{4}(\phi)}{\left(8 i t e^{2 i \phi}\right)^{3}} \widehat{q}_{0}^{(5)}(z_0) \mathrm{d}\phi\\
    &-\int_{-\frac{\pi}{4}}^{0} \frac{\frac{1}{2}\cos{2\phi} e^{4 i t z_{0}^{2}}\cos{\phi}}{\left(-8 i t e^{2 i \phi}\right)^{3}}d\phi\int_{0}^{+\infty} e^{-4ite^{2i\phi}\rho^{2}}\frac{Q(z_0;\rho,\phi)}{\rho^{4}} d\rho \triangleq = \mathbb{I}_4.
\end{flalign}
Where $Q(z_0;\rho,\phi) = \widehat{q}_{0}^{(4)}(t)\rho^{2}\cos^{2}(\phi)-\widehat{q}_{0}^{(4)}(z_0)\rho^{2}\cos^{2}(\phi)-3\widehat{q}_{0}^{(3)}(t)\rho\phi+3\widehat{q}_{0}^{(2)}(t)-3\widehat{q}_{0}^{\prime \prime}(z_0)+\frac{1}{2}\widehat{q}_{0}^{(4)}(z_0)$\\
We compute $\mathbb{I}_4$ as we deal with $\mathbb{I}_3$ and repeat this process,we obtain
\begin{flalign}
    &I_{1}=\int_{-\frac{x}{4}}^{0} \frac{\frac{1}{2}\cos{2\phi} e^{4 i tz_{0}^{2}} }{\left(8 i t e^{2 i \phi}\right)} \widehat{q}_{0}^{\prime}\left(z_{0}\right) d\phi+\sum_{k=2} \frac{1}{(2 k-3) ! !} \int_{-\frac{\pi}{4}}^{0} \frac{\frac{1}{2} \cos{2\phi} e^{4 i t z_{0}^{2}}}{\left(8 i t e^{2 i \phi}\right)^{k}} \widehat{q}^{(2 k-1)}_{0}\left(z_{0}\right) \cos ^{(2k-2)} (\phi)\mathrm{d}\phi \nonumber\\
    &\quad +\sum_{k=1} \int_{-\frac{\pi}{4}}^{0} \frac{\frac{1}{2} \cos{2\phi} e^{4 i t z_{0}^{2}}}{\left(8 i t e^{2 i \phi}\right)^{k}}  \widehat{q}_{0}^{(2 k)}\left(z_{0}\right) \cos ^{(2 k-1)} \phi d\phi \int_{0}^{+\infty} e^{-4 i t \rho^2e^{2 i \phi}} d \rho. \nonumber\\
\end{flalign}
Similarly,we have
\begin{flalign}
    &\mathbb{I}_2 = \sum_{k=1}\frac{1}{(2k)!!}\int_{-\frac{\pi}{4}}^{0}\frac{e^{i\phi} \sin{2\phi} e^{4itz_{0}^{2}}\widehat{q}^{(2k)}_{0}(z_0)\cos^{2k}\phi}{(8ite^{2i\phi})^k}d\phi \int_{0}^{+\infty} e^{-4 i t \rho^2e^{2 i \phi}} d \rho \nonumber\\
    &+\sum_{k=1}\frac{1}{(2k-1)!!}\int_{-\frac{\pi}{4}}^{0}\frac{e^{i\phi} \sin{2\phi} e^{4itz_{0}^{2}}\widehat{q}^{(2k-1)}_{0}(z_0)\cos^{2k-1}\phi}{(8ite^{2i\phi})^k}d\phi.
\end{flalign}
Then $\mathbb{I}_1-i\mathbb{I}_2$ has the following form
\begin{flalign}
    \mathbb{I}_1-i\mathbb{I}_2 = \sum_{k=1}^{n}t^{-k}\widehat{q}^{(2k-1)}_{0}(z_0)e^{4itz_{0}^2}\beta_{k}+\sum_{k=1}^{n}t^{-(2k+1)/2}\widehat{q}^{(2k)}_{0}(z_0)e^{4itz_{0}^2}\gamma_{k}
\end{flalign}
where
\begin{flalign}
    &\beta_{k} = \int_{-\frac{\pi}{4}}^{0}\frac{\frac{1}{2}\cos{2\phi}cos^{2k-2}\phi}{(2k-3)!!(8ie^{2i\phi})^k}-\frac{ie^{i\phi}\sin{2\phi} \cos^{2k}\phi}{(2k)!!(8ie^{2i\phi})^k}d\phi,\\
    &\gamma_{k} = \int_{-\frac{\pi}{4}}^{0}\frac{\frac{1}{2}(\frac{\pi}{i})^{1/2}\cos{2\phi}\cos^{2k-2}(\phi)}{2{\sqrt{\cos{2\phi}}(8ie^{2i\phi})^k}}-\frac{ie^{i\phi}\sin{2\phi}\cos^{2k-1}(\phi)(\frac{\pi}{i})^{1/2}}{2(2k-1)!!\sqrt{\cos{2\phi}}(8ie^{2i\phi})^k}d\phi.
\end{flalign}
Where we define $(2k-3)!! = 1 $ when $k = 1$
It's easy to check the following equation
\begin{equation}
    \int_{\Omega_{+}} \overline{\partial} E(u, v) e^{-2 i t \theta\left(z, z_{0}\right)} dA(u,v) = -\int_{\Omega_{-}} \overline{\partial} E(u, v) e^{-2 i t \theta\left(z, z_{0}\right)} dA(u,v).
\end{equation}
So,the asymptotic expansion of $q(x,t)$ takes the following form
\begin{align}
    q(x,t) &= t^{-1/2}\frac{\widehat{q}_{0}(z_0)e^{-\pi /4}}{2\sqrt{\pi}}e^{ix^{2}/(4t)}+\sum_{k=1}^{n}t^{-k}\widehat{q}^{(2k-1)}_{0}(z_0)e^{4itz_{0}^2}2i\beta_{k}\\
    &+\sum_{k=1}^{n}t^{-(2k+1)/2}\widehat{q}^{(2k)}_{0}(z_0)e^{4itz_{0}^2}2i\gamma_{k}.
\end{align}
\section*{Acknowledgement}
The authors want to thank Prof. Peter Miller from the University of Michigan for helpful comments.
\bibliographystyle{amsplain}

\end{document}